\documentclass[12pt]{amsart}
\usepackage{amsaddr}
\usepackage{tkz-euclide}
\usepackage{latexsym}
\usepackage{amssymb,amsmath,mathrsfs}
\usepackage{amsthm}
\usepackage{bbm}
\usepackage{comment}
\usepackage[capitalise]{cleveref}
\usepackage{enumerate}
\usepackage{esvect}
\usepackage{graphicx}
\usepackage{afterpage}
\usepackage{mwe}
\usepackage{pgfplots}
\pgfplotsset{compat=1.11}
\usepgfplotslibrary{fillbetween}

\allowdisplaybreaks 
\newtheorem{theorem}{Theorem}[section]

\newtheorem{prop}[theorem]{Proposition}
\newtheorem{remark}{Remark}
\newtheorem{lemma}[theorem]{Lemma}

\def\XXint#1#2#3{{\setbox0=\hbox{$#1{#2#3}{\int}$ }
\vcenter{\hbox{$#2#3$ }}\kern-.6\wd0}}

\newcommand{\dist}{\mathrm{dist}}

\newcommand{\F}{f_{LdG}}

\newcommand{\R}{\mathbb{R}}
\newcommand{\N}{\mathbb{N}}

\newcommand{\z}{\hat{z}}
\newcommand{\I}{\mathbf{I}}
\newcommand{\s}{\mathcal{S}}
\newcommand{\f}[2]{\displaystyle\frac{#1}{#2}}
\newcommand\restr[2]{{
  \left.\kern-\nulldelimiterspace 
  #1 
  \vphantom{\big|} 
  \right|_{#2} 
  }}
  \newcommand{\dd}{10\sqrt{\delta}}
\newcommand{\Om}{\Omega}
\newcommand{\rt}{r_{\theta}}
\newcommand{\h}{\mathcal{H}^1}
\newcommand{\pa}{\partial}
\newcommand{\pam}{\partial_M}

\newcommand{\e}{\varepsilon}
\newcommand{\p}{\varphi}
\newcommand{\Q}{\tilde{Q}_\e}
\newcommand{\ls}{\limsup_{\e \rightarrow 0}}

\numberwithin{equation}{section}

\title{Dimension reduction for the Landau--de Gennes model: the vanishing nematic correlation length limit}

\author{Michael R. Novack}
\address{Department of Mathematics, Indiana University}
\email{mrnovack@indiana.edu}
\begin{document}
\date{\today}
\begin{abstract}
 We study nematic liquid crystalline films within the framework of the Landau--de Gennes theory in the limit when both the thickness of the film and the nematic correlation length are vanishingly small compared to the lateral extent of the film. We prove $\Gamma$--convergence for a sequence of singularly perturbed functionals with a potential vanishing on a high--dimensional set and a Dirichlet condition imposed on admissible functions. This then allows us to prove the existence of local minimizers of the Landau--de Gennes energy in the spirit of [R. V. Kohn and P. Sternberg, \textit{Local minimisers and singular perturbations}, Proc. Roy. Soc. Edinburgh Sect. A, 111 (1989), pp. 69--84] despite the lack of compactness arising from the high--dimensional structure of the wells. The limiting energy consists of leading order perimeter terms, similar to Allen--Cahn models, and lower order terms arising from vortex structures reminiscent of Ginzburg--Landau models.
\end{abstract}
\maketitle \centerline{\date}
\section{Introduction}\label{intro}
In this paper, we study thin nematic liquid crystalline films. Nematic liquid crystals consist of rod-- or disk--like molecules which can flow like liquids, but still exhibit molecular orientational order. We analyze the Landau--de Gennes $Q$--tensor variational model in the thin film limit by carrying out a dimension reduction from three dimensions to two. This model exhibits features similar to two--dimensional Allen--Cahn and Ginzburg--Landau theories. We also generalize the local minimizer existence result of \cite{ks89} to functionals with potentials vanishing on high dimensional sets and limiting energy which has interior interfacial cost and cost associated to the boundary of the domain.\par
To arrive at the two--dimensional models, we use the theory of $\Gamma$--convergence; see \cite{dalmaso} for an introduction to the topic. Dimension reduction for planar nematic thin films via $\Gamma$--convergence has also been studied in \cite{gms}. We use a similar model as in \cite{gms} but scale the functional differently by assuming that non--dimensional nematic correlation length vanishes along with the aspect ratio of the film.\par
We fix a bounded, $C^2$ domain $\Omega \subset \R^2$ and work on the cylinder $\Omega \times (0,h)$, where $0 <h \ll 1$. We require admissible $Q$--tensors to have uniaxial boundary data $g$ on the lateral boundary $\pa \Omega \times (0,h)$ of the cylinder. In addition to being uniaxial, $g$ is chosen such that the normal to the top and bottom of the cylinder, in this case  the vector $\z=(0,0,1)$, is an eigenvector. This requirement is motivated by the desire to model both homeotropic and parallel anchoring, in which the majority of the nematic molecules are oriented perpendicular and parallel to the surface of the film, respectively \cite{virg}. Reduced $Q$--tensor models for thin films such as \cite{bpp,sch} have incorporated this condition based on experiments in \cite{chic}.\par
Our starting point is the non--dimensional energy functional
\begin{equation}\notag
E_\e(Q) = \int_{\Om' \times (0,1)} \left(|\nabla_x Q|^2 + \f{1}{\e^2}|\nabla_z Q|^2 + \f{1}{\delta^2}f_{LdG}(Q)\right) \,dx\,dz + \f{1}{\e}\int_{\Om' \times \{0,1\}} f_s(Q) \,dx
,\end{equation}
derived in \cite[Eq. 25]{gms}. The spatial coordinates $(x,z)=(x_1,x_2,z)$ have been rescaled using the change of variables $(x_1,x_2,z)\rightarrow (x_1/R,x_2/R,z/h)$, where $R:=\textup{diam}(\Om)$, and $\Om'=\Om/R$. From now on we will refer to the rescaled domain simply as $\Om$. Here $\e:=h/R$, and $\delta:=\xi_{\mathrm{NI}}/R$ with $\xi_{\mathrm{NI}}$ being the nematic correlation length \cite{gartland}. The nondimensional elastic energy density $|\nabla_x Q|^2+\f{1}{\e^2}|\nabla_z Q|^2$ corresponds to the so--called equal elastic case of the general Landau--de Gennes elastic energy density. The $f_{LdG}$ term is the nondimensional Landau--de Gennes bulk energy density \cite{gartland}; see Section \ref{pre} for a precise definition.

The surface term $f_s$ from \cite{gms} is given by
\begin{equation}\notag
f_s(Q) =  \gamma|(\I - \z \otimes \z )Q\z|^2+\alpha(Q\z \cdot \z - \beta)^2
,\end{equation}
where $\alpha,\gamma>0$, $\beta$ is real, and $\I$ is the $3\times 3$ identity matrix. It is evident from this expression that the minimizers of $f_s$ are precisely those $Q$--tensors which have $\z$ as an eigenvector with eigenvalue $\beta$. This surface energy is a specific example of the ``bare" surface energy
\begin{equation}\notag
f_s(Q)=c_1(Q\z \cdot \z)+c_2|Q|^2 + c_3 (Q\z \cdot \z)^2+c_4 |Q\z|^2,
\end{equation}
for a planar film, where $\alpha=c_3+c_4>0$, $\beta = -\f{c_1}{2(c_3+c_4)}$, $c_2=0$, and $\gamma=c_4>0$. Choosing the constants $c_i$ to satisfy these constraints ensures that the energies $F_\e$ are bounded from below and also allows for nontrivial behavior in the limit due to the degeneracy of the set of minimizers for $f_s$.\par
Next, we wish to choose a suitable scaling for $\delta$ in terms of $\e$ and carry out the one-parameter dimension reduction. In order for both the Landau--de Gennes and surface terms to play a role in the asymptotic regime, a natural scaling to consider is $\delta\sim\sqrt\e$ when the core radius of nematic defects, although small, is still much larger than the thickness of the film. Then the energy $E_\e$ can be expressed in the form
\begin{equation}\notag
E_\e(Q) = \int_{\Omega \times (0,1)}  \left( |\nabla_xQ|^2 +\f{|\nabla_{z} Q|^2}{ \e^2 } +  \f{1}{\e}f_{LdG}(Q) \right)\,dx\,dz + \f{1}{\e}\int_{\Omega \times \{0,1\}} f_s(Q) \,dx
.\end{equation}

With this scaling in place, we consider an asymptotic limit of the functionals $E_\e$ as $\e\to0$. This is in contrast to \cite{gms}, where only the non--dimensional thickness of the film was assumed to approach zero.

For convenience, we replace $\e$ by $\e^2$ and multiply the modified functionals by $\e$. This results in functionals for which the leading con
contribution will be $O(1)$:
\begin{equation}\notag
F_\e(Q) =  \int_{\Omega \times (0,1)} \left(\e|\nabla_xQ|^2 + \f{|\nabla_{z} Q|^2}{ \e ^3} + \f{1}{\e}f_{LdG}(Q) \right)\,dx\,dz+\f{1}{\e} \displaystyle \int_{\Omega \times \{0,1\}} f_s(Q) \,dx.
\end{equation}
Heuristically, due to the high cost associated with $z$--dependence, reasonable competitors for $F_\e$ should be essentially independent of $z$. This observation allows us to rewrite the integral of $f_s$ over $\Omega\times \{0,1\}$ as twice the integral of $f_s$ over $\Om \times(0,1)$. With the equal scaling of the Landau--de Gennes and the surface terms then, we obtain a potential which is $f_{LdG}$ perturbed by $2f_s$. Due to the growth of $f_{LdG}$ and $f_s$ at infinity, the minimum value of $f_{LdG}+2f_s$ is achieved and then, without loss of generality, we can set the minimum to be zero. Denoting the zero set of $\F+2f_s$ by $P$, we see that the bulk and surface terms penalize those $Q$ which take values away from $P$.\par
Let us denote by $P_i$ a connected component of $P$ and define the function $\p_i(Q)$ as \begin{equation}\notag d_{\sqrt{W}}(Q,P_i),\end{equation} where $d_{\sqrt{W}}$ is the degenerate Riemannian metric with the conformal factor $\sqrt{f_{LdG}+2f_s}$. Also, we will frequently identify maps on $\Om$ as maps defined on $\Om \times (0,1)$ by the trivial $z$--independent extension. Our first main theorem is:
\begin{theorem}\label{t1vague}
The $F_\e$ $\Gamma$--converge in the $L^1$ topology to the functional $F_0$, defined by
\begin{equation}\notag
F_0(Q)= \sum_{i,j=1}^n \p_i(P_j) \mathcal{H}^1(\partial^* A_i \cap \partial^* A_j) + 2\sum_{i=1}^n \int_{\pa A_i \cap \pa \Om} \p_i(g(x)) \,d\mathcal{H}^1(x),
\end{equation}
where $A_i:=\{x \in \Om: Q(x) \in P_i\}$, the set $\partial^* A_i$ is the reduced boundary of $A_i$, cf. \cite{giusti}, and $\h$ denotes the one--dimensional Hausdorff measure.
\end{theorem}\par
The $\Gamma$--convergence of Allen--Cahn type functionals is well-studied in the literature; see for example \cite{baldo,fons,ks89,mm,sternbergthesis,rocky}. In each of these examples, the zero set of the potential function is a finite number of points. In our case, however, the minimal set of the potential is the zero set of a quartic polynomial in five variables, and we do not have a full description of its structure. Generally, the wells are high--dimensional and perhaps contain singularities. \color{black}The $\Gamma$-- convergence of Allen--Cahn type functionals possessing a general potential function with a zero set of arbitrary dimension and smoothness has been established in \cite{ambrosio}. The model we consider is distinct from \cite{ambrosio} in that it combines a high--dimensional potential well with a Dirichlet condition. \color{black}We also point out that it has been more common in the study of such problems to impose a volume constraint on the space of admissible functions, as in \cite{sternbergthesis,baldo,fons}, rather than a Dirichlet condition, as considered here.\par
For the scalar case, $\Gamma$--convergence of Allen--Cahn functionals with competitors satisfying a Dirichlet condition was proved in \cite{owen}. The asymptotics of vector--valued minimizers of Allen--Cahn functionals, rather than full $\Gamma$--convergence, in the presence of a suitable Dirichlet condition has been addressed in \cite{lin}. To the best of our knowledge, these appear to be a few of the rare instances of a discussion of $\Gamma$--convergence for singular perturbations of multi--well potentials among maps subject to a Dirichlet condition. This may be due to the inherent difficulty in building a boundary layer consisting of a smoothly varying family of geodesics bridging the Dirichlet condition lying off the potential wells to values in the wells.

One of the main contributions of the present work is obtaining full $\Gamma$--convergence for Allen--Cahn type functionals with higher--dimensional wells among competitors that satisfy a Dirichlet condition. The rotational symmetry described in Section \ref{gproof} of the Landau--de Gennes potential $f_{LdG}$ and the surface term $f_s$ is crucial to the proof of $\Gamma$--convergence, in particular, to the construction of a recovery sequence. This symmetry was also utilized by the authors in \cite{abx2} in the context of studying a nematic liquid crystal outside a spherical particle under an external field. Away from the boundary, we are able to apply the techniques from \cite{ambrosio}. We point out that the techniques in the $\Gamma$--convergence proof would apply equally well to similar functionals with a potential vanishing on a higher--dimensional set, as long as the potential has some symmetry respected by the Dirichlet boundary data; see the discussion at the end of Section \ref{local}.\par \color{black}
A natural question regarding Allen--Cahn type functionals such as $F_\e$ is the possibility of finding local minimizers for the functionals $F_\e$. In cases such as the one discussed here, where the zero set of the potential consists of curves or even surfaces as opposed to isolated points, this has not been addressed in the literature. When the zero set consists of two points, an answer was provided by the authors in \cite{ks89}. Crucial to their proof is the $L^1$--compactness of a bounded energy sequence, which cannot be expected in general when the zero set is higher--dimensional. We resolve this question for very general functionals by working with a distance which is weaker than the usual $L^1$ metric. For two tensors $Q_1$ and $Q_2$, we define
\begin{equation}\notag
\Lambda(Q_1,Q_2) = \sum_i\|\p_i(Q_1)-\p_i(Q_2) \|_{L^1(\Om)}.\end{equation}
For $Q$--tensors whose range is contained in the zero set $P=\cup_i P_i$, the distance $\Lambda(Q_1,Q_2)$ is zero precisely when the sets where $Q_1$ and $Q_2$ lie in $P_i$ coincide for each $i$. It can be quickly shown (see Section \ref{local}) that a sequence $\{Q_\e\}$ with bounded energy is $\Lambda$--compact. A similar pseudo--metric was proposed in \cite{ambrosio}, under which bounded energy sequences enjoy compactness. \color{black}This also allows us to prove the existence of local minimizers of $F_\e$ and $F_0$ in Sections \eqref{local} and \eqref{iso} similar to those in \cite{ks89} and \cite{ziemer} for potentials that vanish on sets more complicated than a finite collection of points. The theorems apply to our particular liquid crystal models as well as a wide range of other functionals. Phrasing the rather general result in terms of our specific liquid crystal problem, we prove
\begin{theorem}\label{minivague}
Let $Q_0$ be an isolated $\Lambda$--local minimizer of $F_0$. Then there exists $\e_0>0$ and a family $\{Q_\e \}_{\e<\e_0}$ such that
\begin{equation}\label{min1vague}
Q_\e \text{ is a }\Lambda \text{--local minimizer of } F_\e
\end{equation}
and
\begin{equation}\label{min2vague}
\Lambda(Q_0,Q_\e) \rightarrow 0
.\end{equation}
\end{theorem}
It can be quickly shown (see Remark \ref{mike}) that the $Q_\e$ are in fact $H^1$--local minimizers, and therefore are classical solutions, by elliptic regularity, to the Euler-Lagrange system corresponding to $F_\e$.\par
In order to apply \cref{minivague}, we must obtain a $\Lambda$--isolated local minimizer for $F_0$, which has interior interfacial cost as well as cost associated to the boundary of $\Om$. We prove the existence of such a local minimizer in certain domains by using a calibration argument. The interface of the locally minimizing partition is contained in a ``neck" of such a domain. However, unlike for example the situation in \cite{ks89}, the optimal interface is a straight line which does \textit{not} meet the boundary of $\Om$ at the narrowest part of the neck. Rather, the interface satisfies a contact angle condition, given by \eqref{contact}, in which the boundary energy is balanced by the interfacial energy. \par
We also wish to study the formation of defects for equilibrium configurations of liquid crystalline films. Such defects often take the form of disclination lines, which arise due to the degree of the Dirichlet boundary data $g$. The history of these types of questions goes back to \cite{bbh}. More recently, for liquid crystalline films, the authors in \cite{bpp} considered such defects for a two--dimensional Landau--de Gennes model. The study of defects, or vortices, in the presence of boundary layers for Ginzburg--Landau type energies has been analyzed in \cite{as1,as2,as3}. In our scaling, the energy associated to these types of defects should be on the order of $\e |\log \e |$. This is lower order than the $O(1)$ cost associated to any boundary layers for finite energy $Q_\e$. Although the maps under consideration here are $\R^5$--valued, as opposed to $\R^2$--valued, the main techniques of these papers are applicable to our problem after some adjustments. In the following theorem, we use the techniques from \cite{as3} to obtain an asymptotic expansion for the energy of a minimizer $Q_\e$ of $F_\e$ for a specific choice of the surface energy density $f_s$. 
\begin{theorem}\label{asyv1} Let $\Omega$ be a simply--connected domain and $g$ have degree $k$. Assume that \begin{equation}\notag f_s=  \gamma|(\I - \z \otimes \z )Q\z|^2.\end{equation} Then the minimizers $Q_\e$ of $F_\e$  satisfy the asymptotic development
\begin{equation}\notag
F_\e(Q_\e) = 2 \int_{\partial \Omega} \p_1(g(x)) \, d\h(x) + \e s_\star^2  \pi k \log\f{1}{\e} + O(\e)
\end{equation}
as $\e\rightarrow 0$. The constant $s_\star$ is explicit and depends on the bulk term $\F$.
\end{theorem}
We also briefly discuss the convergence of $Q_\e$ to a limiting map resembling the canonical harmonic map from \cite{bbh}, as was done in \cite{bpp}, and the location of the defects governed by a certain ``renormalized energy."\par
The paper is organized as follows. In Section \ref{pre} we introduce the problem in full detail. We also present some preliminaries which will be necessary in the proofs of our results. In Section \ref{gproof} we state the exact version of the $\Gamma$--convergence result, \cref{t1vague}, and give its proof. In Sections \ref{local} and \ref{iso} we prove the existence of local minimizers of $F_\e$ and $F_0$, respectively. In Section \ref{degree} we analyze the issue of defects. Rather than giving full proofs for every result in that section, we outline the ideas and refer the reader to \cite{as3,bpp} for the specific calculations. Finally, in the appendix, we give a partial characterization of the zero set of the modified potential by establishing conditions under which $\z$ is an eigenvector of any minimizer of $\F+2f_s$.
\section{Notation and Preliminaries}\label{pre}
First, we present a brief outline of the Landau--de Gennes $Q$--tensor theory. This theory is built on the $Q$--tensor order parameter field corresponding to the second moment of the local orientational probability distribution (see \cite{mz,newt}). $Q$--tensors are symmetric, traceless, $3\times3$ matrices which are used to model the nematic state of a liquid crystal. As such, they have an orthonormal basis of eigenvectors $v_i$ and corresponding real eigenvalues $\lambda_i$. When two of the eigenvalues of the $Q$--tensor are equal, for example $\lambda_1=\lambda_2$, then the nematic liquid crystal is in a \textit{uniaxial} state and the $Q$--tensor can be written as
\begin{equation}
Q= S\left( v_3 \otimes v_3 - \f{1}{3} \I \right),
\end{equation}
where $S=3\lambda_3/2$, and $\I$ is the identity. If no two eigenvalues of $Q$ are equal, then the liquid crystal is in a \textit{biaxial} state, and
\begin{equation}
Q=S_1\left( v_1 \otimes v_1 - \f{1}{3} \I \right)+S_2\left( v_3 \otimes v_3 - \f{1}{3} \I \right)
,\end{equation}
where $S_1=2\lambda_1+\lambda_3$ and $S_2=\lambda_1+2\lambda_3$. The biaxial state differs from the uniaxial state in that it only possesses reflection symmetries with respect to the three orthogonal axes, whereas the uniaxial state possesses rotational symmetry. Biaxial states may in particular exist in the cores of nematic defects \cite{three}. When all three eigenvalues are zero, the $Q$--tensor is in the \textit{isotropic} state.\par
The associated variational model we consider involves minimization of an energy functional composed of elastic-- and bulk--volume terms as well as a weak anchoring surface term. The most general elastic term one might use includes quadratic and even cubic terms in $Q$ and its derivatives \cite{newt}. Throughout this paper, however, we will work in the so--called equal elastic constants regime which corresponds to the usual Dirichlet energy. For the bulk energy density, we will use the Landau--de Gennes bulk energy density given by
\begin{equation}
\label{eq:ldgfun}
f_{LdG}(Q):= a \textup{tr}(Q)^2+ \f{2b}{3}\textup{tr}(Q^3)+\f{c}{2}(\textup{tr}(Q^2))^2
.\end{equation}
The coefficient $a$ depends on temperature and is negative for sufficiently low temperatures, while $c>0$. One quickly sees that $f_{LdG}$ only depends on the eigenvalues of $Q$. It can be shown that for high temperatures, the global minimum of $f_{LdG}$ is attained by the isotropic state, whereas for low temperatures, the global minimum is achieved by the uniaxial state
\begin{equation}\label{state}
s_\star\left ( m \otimes m - \f{1}{3}\I \right),
\end{equation}
where $m\in \mathbb{S}^2$ and $s_\star$ is a real number given explicitly in terms of $a$, $b$, and $c$ \cite{newt} . Throughout most of this paper, we do not impose any assumptions on the temperature. Note however that since $c>0$, $f_{LdG}$ is bounded from below, and moreover, it grows quartically at infinity. We refer the reader to \cite{gartland} for a thorough overview of the proper non--dimensionalization procedure for \eqref{eq:ldgfun}. 

Recall from the introduction that the weak anchoring surface term is given by:
\begin{equation}
f_s(Q)=\gamma|(\I - \z \otimes \z )Q\z|^2+\alpha(Q\z \cdot \z - \beta)^2
.\end{equation}
From the perspective of modeling nematics, the constant $\beta$ satisfies $-1/3 < \beta < 2/3$; see \cite[p. 24]{virg}. However, our arguments hold for any real $\beta$. Models with weak anchoring surface terms have been a source of much recent research, for example in \cite{abx1,ball,cane,gm,mz,seg}.\par
Next, let us introduce in full detail the problem under consideration. Let $\s$ be the set of real, symmetric, traceless, $3 \times 3$ matrices. We assume $\Omega$ is a bounded domain in $\R ^2$ with $C^2$ boundary. For $(x,z)\in \pa \Om \times (0,1)$, let $g$ be a Lipschitz function given by
\begin{equation}\label{g1}
g(x,z)=\f{3 \beta}{2} \left( \z \otimes \z - \f{1}{3} \I \right) 
\end{equation}
or
\begin{equation}\label{g2}
g(x,z) = -3\beta \left( n(x) \otimes n(x) - \f{1}{3} \I \right) 
,\end{equation}
where $n$ is $\mathbb{S}^1\times \{0\}$--valued vector field and does not depend on $z$. We remark that in both cases, $g$ does not depend on $z$, so we will refer to $g$ as a function of $x$ from now on. Since minimizers of $f_s$ must have $\z$ as an eigenvector with eigenvalue $\beta$, we see that these are the only uniaxial minimizers of $f_s$. Let  $$\mathcal{A}_g=\{Q \in H^1: \restr{Q}{\partial \Omega \times (0,1)}=g\}.$$
Define $W:\s \rightarrow \R$ by
\begin{equation}\label{W}
W(Q)=f_{LdG}(Q)+2f_s(Q),
\end{equation}
and let us suppose that $W(g)$ is never $0$; if it were, the problem would be trivial to zero order. By the growth of $f_{LdG}$ at infinity, we know that $W$ achieves a global minimum value. Adding a constant to $W$, we can assume without loss of generality the minimum of $W$ is $0$.\par
Recall that
$$ F_\e(Q) = \begin{cases} 
      \displaystyle \int_{\Omega \times (0,1)} \left(\e|\nabla_xQ|^2 + \f{|\nabla_{z} Q|^2}{ \e ^3} + \f{1}{\e}f_{LdG}(Q)\right) \,dx\,dz \\ \hspace{.6cm}+\f{1}{\e} \displaystyle \int_{\Omega \times \{0,1\}} f_s(Q) \,dx,   & Q \in \mathcal{A}_g,\\
      \infty, & \text{otherwise,}
   \end{cases}
$$
for every $\e>0$. Let $P=\{Q \in \s : W(Q)=0\}$; then $P=\cup_{i=1}^n P_i$, where the $P_i$'s are the connected components of $P$. It is known that $n$ is finite by a theorem of Whitney \cite{whitney}. In the appendix we prove a result showing when $\z$ is an eigenvector for any element of $P$. Note that the growth of $W$ at infinity implies that each $P_i$ is compact.\par
We now introduce the limiting functional $F_0$ for the sequence of functionals $\{F_\e\}$. First, define for any two subsets $U,V$ of $\s$
\begin{align}
\label{asdf} d_{\sqrt{W}}(U,&V):=  \\ \notag \inf&\left\lbrace \displaystyle\int_a^b  \sqrt{W(\gamma(t))}|\gamma'(t)| \,dt \textrm{ s.t. }\gamma:[a,b]\rightarrow \s \text{ is Lipschitz},\gamma(a)\in U, \gamma(b)\in V \right\rbrace.
\end{align}
Note that the above integral is independent of parametrization. For any $Q:\Om \rightarrow P$, we set $A_i=\{x: Q(x)\in P_i\}$ and \begin{equation}\label{pii}\p_i(Q):=d_{\sqrt{W}}(P_i,Q).\end{equation}Define \begin{equation}\label{a00}\mathcal{A}_0 = \left\lbrace  Q\in L^1(\Om;\s) :  W(Q(x))=0 \textrm{ a.e., }\textrm{and }\mathbbm{1}_{A_i} \in BV(\Om) \right\rbrace .\end{equation}
For maps in $\mathcal{A}_0$, we will view them as maps defined on the cylinder $\Om \times (0,1)$ via the trivial extension to three dimensions. We will similarly view maps defined on the cylinder as defined on $\Om$ if they do not depend on $z$. We then define our candidate for the $\Gamma$--limit $F_0$ by 
\begin{equation}\label{f0intro}
F_0(Q) := \begin{cases} 
     \sum_{i,j=1}^n \p_i(P_j)\mathcal{H}^1(\partial^* A_i \cap \partial^* A_j) & Q \in \mathcal{A}_0,\\ \hspace{.6cm} +2\sum_{i=1}^n \int_{\pa A_i \cap \pa \Om} \p_i(g(x)) \,d\mathcal{H}^1(x) \\
      \infty & \text{otherwise.}
   \end{cases}
\end{equation}\par
\color{black}Finally, we present some preliminaries necessary for the proof of \cref{t1vague}. Following \cite{ambrosio}, we recall the definition of a $BV$ function taking values in a locally compact metric space $E$. For an arbitrary open set $D\subset R^3$, the space $BV(D,E)$ is the class of Borel functions $v:D \to E$ such that for any Lipschitz $\psi:E\to \R$,
\begin{align*}
\psi(v) \in BV(D;\R)
\end{align*}
and there exists a finite measure $m$ satisfying
\begin{align}\label{measure}
|\nabla (\psi \circ v)|(B) \leq \textup{Lip}(\psi) m(B)\textup{ for all Borel sets }B\subset D.
\end{align}
Here $|\nabla (\psi \circ v)|$ is the usual total variation measure for real--valued $BV$ functions. The total variation measure $|\nabla v|$ of $v$ is the least such measure satisfying \eqref{measure}. For our purposes, $E$ will be the canonical quotient space of $(\R^5,d_{\sqrt{W}})$, with distance function in $E$ also denoted by $d_{\sqrt{W}}$. In this space, each $P_i$ is identified with a single point. We will need the following form of lower--semicontinuity in $BV(D,E)$: if $d_{\sqrt{W}}(v_\e,v) \to 0$ in $L^1$ and $v_\e \in BV(D,E)$, then $v \in BV(D,E)$ and
\begin{align}\label{abslsc}
|\nabla v|(D) \leq \liminf_{\e \to 0} |\nabla v_\e|(D).
\end{align}
\par
The last preliminary result is a lemma from \cite{baldo} regarding approximating a partition of $\Om$ by polygonal domains and will be used in our construction of a recovery sequence.\color{black}
\begin{lemma}\label{poly}
\textup{(cf. \cite[Lemma 3.1]{baldo}).} Let $A_i$, $1\leq i \leq n$, be a partition of $\Om$ by sets of finite perimeter. There exists a sequence $\{A_1^h, \dots , A_n^h\}_{h \in \N}$ of collections of polygons such that
\begin{enumerate}[(i)]
\item $\Om \subset \cup_i A_i^h$ and the $A_i^h$ are pairwise disjoint up to sets of measure zero;
\item $\h (\pa A^h_i\cap \pa \Om)=0$;
\item $|A_i^h \triangle A_i| \rightarrow 0$ as $h \rightarrow \infty$;
\item $\sum_{i,j=1}^n \p_i(P_j)\mathcal{H}^1(\partial^* A_i^h \cap \partial^* A_j^h\cap \Om)  \rightarrow \sum_{i,j=1}^n \p_i(P_j)\mathcal{H}^1(\partial^* A_i \cap \partial^* A_j\cap \Om)$ as $h \rightarrow \infty$.
\end{enumerate}
\end{lemma}
\section{Proof of $\Gamma$--convergence}\label{gproof}
In this section, we prove 
\begin{theorem}\label{t1}
For either choice \eqref{g1} or \eqref{g2} of boundary data $g$, the sequence $\{F_\e\}$ $\Gamma$--converges in the topology of $L^1(\Om\times(0,1);\s)$ to $F_0$. That is, 
\begin{enumerate}
\item for any $Q\in L^1(\Om\times(0,1);\s)$ and for any sequence $\{Q_\e\}$ in $L^1(\Om\times(0,1);\s)$, 
\begin{equation}\label{lsc}
	Q_\e \rightarrow Q \text{ in } L^1(\Om\times(0,1);\s) \text{ implies } \liminf_{\e \rightarrow 0} F_\e(Q_\e) \geq F_0(Q),
\end{equation} \label{statement1}
and
\item for each $Q\in L^1(\Om\times(0,1);\s)$ there exists a recovery sequence $\{Q_\e\}$ in $L^1(\Om\times(0,1);\s)$ satisfying 
\begin{equation}\label{reco1}
Q_\e \rightarrow Q_0 \text{ in } L^1(\Om\times(0,1);\s),
\end{equation}
\begin{equation}\label{reco2}
\lim_{\e\rightarrow 0} F_\e(Q_\e)= F_0(Q_0).
\end{equation} \label{statement2}
\end{enumerate}
\end{theorem}

For the rest of this section, we fix $g$ as specified in \eqref{g2} and proceed with the proof for this choice of $g$. The proof for $g$ given by \eqref{g1} is omitted since it is similar but simpler due to the constancy of $g$. Throughout the proofs, we will make use of the following symmetry of $W$ and its zero set $P$. 
Fix any $\theta \in [0,2\pi)$ and consider the rotation matrix
\begin{equation}\notag
 r_{\theta}=
  \left[ {\begin{array}{ccc}
   \cos(\theta) & -\sin(\theta) & 0 \\
   \sin(\theta) & \cos(\theta) & 0 \\
   0 & 0 & 1 \\
  \end{array} } \right].
\end{equation}
For any $Q$--tensor, which we can write as $Q=\sum_{i=1}^3\lambda_i v_i \otimes v_i$ in some orthonormal frame $\{v_i\}$,  we have 
\begin{equation}\notag
\rt Q \rt^T= \sum_{i=1}^3 \lambda_i \rt v_i \otimes \rt v_i.
\end{equation}
Recall that $f_{LdG}$ only depends on the eigenvalues $\lambda_i$ of $Q$, and so the preceding equality implies that $f_{LdG}(Q)=f_{LdG}(\rt Q \rt^T)$. Also,  for $f_s(Q)$, using the facts $\z=\rt \z = \rt^T \z$, $(\I - \z \otimes \z )\rt = \rt (\I - \z \otimes \z )$, and $|\rt v|=|v|$ for any $v\in \R^3$, we can calculate
\begin{equation}\notag
f_s(\rt Q \rt^T) = f_s(Q).
\end{equation}
It follows that $W(Q)=f_{LdG}(Q)+2f_s(Q)$ remains unchanged after conjugating $Q$ by $\rt$. From this we deduce that for any connected component $P_i$ of $P$,  conjugating $P_i$ by $\rt$ fixes $P_i$, i.e., 
\begin{equation}\label{rot}
\rt P_i \rt^T = P_i.
\end{equation}
We remark that this equation implies that any well $P_i$ of $W$ will in general be at least $1$-dimensional. Finally, note that conjugation by $\rt$ preserves the quantity $|Q|^2:=tr(QQ^T)$, the sum of the squares of the entries of $Q$. Using this observation and the preceding comments regarding $W$, we deduce that for any piecewise $C^1$ path $\gamma:[a,b]\rightarrow \s$, 
\begin{equation}\label{yay}
\int_a^b \sqrt{W(\gamma(t))}|\gamma'(t)| \,dt = \int_a^b \sqrt{W(\rt \gamma(t)\rt^T)}|(\rt\gamma\rt^T)'(t)| \,dt.
\end{equation}
We can use \eqref{yay} to show that for any $g$ satisfying \eqref{g2} and any $1\leq i \leq n$, the distance $d_{\sqrt{W}}(P_i,g(x))$ defined in \eqref{asdf} does not depend on the choice of $(x,z)$ in $\pa \Om \times (0,1)$. The previous equality implies that for any $(x,z)\in \pa \Om \times (0,1)$ and $\theta \in [0,2\pi)$,
\begin{equation}\notag
d_{\sqrt{W}}(P_i,g(x)) = d_{\sqrt{W}}(\rt P_i \rt^T, \rt g(x) \rt^T).
\end{equation}
But since $\rt P_i \rt^T = P_i$, we conclude that
\begin{equation}\label{charlie}
d_{\sqrt{W}}(P_i,g(x))=d_{\sqrt{W}}(P_i,\rt g(x) \rt^T)
\end{equation}
For any fixed $(x,z)$ in $\pa \Om \times (0,1)$, the set $P_0=g(\pa \Om \times (0,1))$, is contained in $\{\rt g(x) \rt^T:0 \leq \theta \leq 2\pi\}$, so in fact \eqref{charlie} yields
\begin{equation}\label{eq}
d_{\sqrt{W}}(P_i,P_0)=d_{\sqrt{W}}(P_i,g(x)).
\end{equation}\par
Let us now proceed with the proof of \cref{t1}. Throughout the estimates, whenever appropriate, the generic constant $C$ varies from line to line.
\begin{proof}[Proof of lower semicontinuity \eqref{lsc}]
Let $Q_\e$ converge to $Q$ in $L^1(\Omega\times(0,1);\s)$, and fix a subsequence $\e_m$ such that $\lim_{\e_m\rightarrow 0}F_{\e_m}(Q_{\e_m}) = \liminf_{\e\rightarrow 0}F_\e(Q_\e)$. We may without loss of generality assume that the limit $F_{\e_m}(Q_{\e_m})$ is finite and that the energies $F_{\e_m}(Q_{\e_m})$ are uniformly bounded by some $C$. Using this assumption we can make a calculation which will simplify the proof of lower semicontinuity by replacing integrals of $f_s$ over the top and bottom of the cylinder by the integral of $2f_s$ over the whole cylinder $\Om \times (0,1)$. For the matrix $Q_{\e_m}$, we denote the entry in the $i$--th row and $j$--th column by $q^{\e_m}_{ij}$. In the following calculation, it is helpful to rewrite $f_s(Q)$ as $f_s(Q)=\gamma(q_{13}^2+q_{23}^2)+\alpha(q_{33}-\beta)^2$. Now, the uniform energy bound implies that
\begin{align}\label{bd1}
\int_{\Omega\times(0,1)}|\nabla_{z}Q_{\e_m}|^2 \,dx\,dz 
&\leq C{\e_m}^3
\end{align}
and
\begin{equation}\label{bd2}
\int_{\Omega\times\{0,1\}}(\gamma(q_{13}^2+q_{23}^2)+\alpha(q_{33}-\beta)^2) \,dx\,dz 
\leq C{\e_m}.
\end{equation}
Using these two estimates it follows that 
\begin{align}
\notag \f{1}{{\e_m}}\left| \int_{\Omega\times(0,1)} f_s(Q_{\e_m})\right. & \left. \,dx\,dz -\int_{\Omega\times\{0\}}f_s(Q_{\e_m}) \,dx \right|  \\
 &= \f{1}{{\e_m}}\left|\int_{(0,1)}\int_{\Omega}(f_s(Q_{\e_m})-f_s(Q_{\e_m}(x,0))) \,dx \,dz \right| \notag \\
 &= \f{1}{{\e_m}}\left| \int_{(0,1)}\int_{\Omega}\int_0^{z} \f{\partial}{\partial z} f_s(Q_{\e_m}(x,t)) \,dt \,dx \,dz \right| \notag \\
 &\leq \f{C}{{\e_m}} \int_{\Omega}\int_0^1 \left|2q^{\e_m}_{13}\f{\partial}{\partial z}q^{\e_m}_{13}\right|+\left|2q^{\e_m}_{23}\f{\partial}{\partial z}q^{\e_m}_{23}\right|\notag +\left|2(q^{\e_m}_{33}-\beta)\f{\partial}{\partial z}q^{\e_m}_{33}  \right| \,dt \,dx \notag \\
 &\leq \f{C}{{\e_m}}\|q^{\e_m}_{13}\|_{L^2}\left\Vert\f{\partial}{\partial z}q^{\e_m}_{13} \right\Vert_{L^2}+\f{C}{{\e_m}}\|q^{\e_m}_{23}\|_{L^2}\left\Vert\f{\partial}{\partial z}q^{\e_m}_{23} \right\Vert_{L^2}\\ \notag &\hspace{1cm}+\f{C}{{\e_m}}\|q^{\e_m}_{33}-\beta\|_{L^2}\left\Vert\f{\partial}{\partial z}q^{\e_m}_{33}\right\Vert_{L^2}\\
 &\leq C\e_m. \label{decay1}
\end{align}
A similar estimate holds for the other surface term:
\begin{equation}\label{decay2}
\f{1}{{\e_m}}\left| \int_{\Omega\times(0,1)} f_s(Q_{\e_m}) \,dx\,dz -\int_{\Omega\times\{1\}}f_s(Q_{\e_m}) \,dx \right|\leq C\e_m
.\end{equation}
Together these estimates imply that in order to prove lower semicontinuity for the sequence $Q_{\e_m}$, we can replace $F_{\e_m}(Q_{\e_m})$ by 
\begin{equation}\label{fe}
\tilde{F}_{{\e_m}}(Q_{\e_m}):=\int_{\Omega\times(0,1)} \left({\e_m}|\nabla_xQ_{\e_m}|^2 + \f{|\nabla_{z} Q_{\e_m}|^2}{ {\e_m} ^3} + \f{1}{{\e_m}}W(Q_{\e_m})\right)\,dx\,dz
\end{equation}
where, as usual, $W$ is given by \eqref{W}, and prove the lower semicontinuity inequality for these energies.\par
Next, in order to capture the cost associated with $\partial \Omega\times(0,1)$ coming from the transition layer from the boundary data $g$, we will need to consider a smooth domain $\tilde{\Omega}$ containing $\Om$ and extend the domain of definition of each $Q_{\e_m}$ to $\tilde{\Omega}\times(0,1)$. We do this by setting each $Q_{\e_m}$ equal to $g(x,z)$ on the segment normal to $\pa \Om\times(0,1)$ at $(x,z)$. Note this extension preserves the $H^1$ regularity possessed by $Q_{\e_m}$. We also perform the same extension on $Q$. Restricting to a further subsequence (still denoted $Q_{\e_m}$) which converges pointwise to $Q$, we use Fatou's Lemma and the preceding remark to obtain:
$$
\int_{\Omega\times(0,1)}W(Q)\,dx\,dz \leq \liminf_{m\rightarrow \infty}\int_{\Omega\times(0,1)}W(Q_{\e_m})\,dx\,dz \leq \liminf_{m  \rightarrow \infty}\e_{m} \tilde{F}_{\e_m}(Q_{\e_m}) = 0;
$$
hence $W(Q)=0$ a.e. in $\Omega \times (0,1)$. Finally, we claim that we can truncate $Q_{\e_m}$, in the sense of truncating each component at a certain size, and reduce all of the energies $\tilde{F}_{\e_m}(Q_{\e_m})$ as well as preserve the $L^1$ convergence of $Q_{\e_m}$. The energies decrease if we truncate far enough from $0$ because of the quartic growth at infinity of $W$. The $L^1$ convergence is preserved because $Q$ must be bounded, given that the zero set of $W$ is compact.
\par
\color{black}Now, utilizing \cite[Proposition 4.3]{ambrosio}, we find that $\pi\circ Q_{\e_m} \in BV(\tilde{\Omega}\times(0,1),E)$, where $E$ is the quotient of $(\R^5,d_{\sqrt{W}})$ and $\pi$ is the projection from $\R^5$ onto $E$. Furthermore, according to \cite[Equation 4.6]{ambrosio}, $\pi\circ Q_{\e_m}$ satisfies the bound
\begin{align}\notag
|\nabla (\pi \circ Q_{\e_m}) |(\tilde{\Omega}\times(0,1)) &\leq \int_{\tilde{\Omega}\times(0,1)}\sqrt{W(Q_{\e_m})}|\nabla Q_{\e_m}| \,dx\,dz. \end{align}
Splitting the right hand side into two integrals and multiplying by $2$, we estimate
\filbreak\begin{align}\notag
2|\nabla (\pi \circ Q_{\e_m}) |(\tilde{\Omega}&\times(0,1))\notag\\ \notag \nopagebreak[4]&\leq  2\int_{\Omega\times(0,1)} \sqrt{W(Q_{\e_m})}|\nabla Q_{\e_m}| \,dx\,dz +2\int_{(\tilde{\Omega}\setminus \Omega)\times(0,1)} \sqrt{W(Q_{\e_m})}|\nabla Q_{\e_m}| \,dx\,dz \notag \nopagebreak[4]\\
 &\leq  \tilde{F}_{\e_m}(Q_{\e_m})+2\int_{(\tilde{\Omega}\setminus \Omega)\times(0,1)} \sqrt{W(Q_{\e_m})}|\nabla Q_{\e_m}| \,dx\,dz \label{i2}
\end{align}
We observe that by the definition of the extension of $Q_{\e_m}$ to $\tilde{\Omega}\times(0,1)$, the integrand of second integral is bounded by constant $C(g)$ depending only the boundary data $g$.\par
Next, since $d_{\sqrt{W}}$ is bounded above by the Euclidean distance on compact sets and the $Q_{\e_m}$ are bounded, we have that $d_{\sqrt{W}}(\pi\circ Q_{\e_m},\pi \circ Q) \rightarrow 0$ in $L^1$. Then the lower--semicontinuity property \eqref{abslsc} in $BV(\tilde{\Omega}\times(0,1),E)$ yields
\begin{align}\notag
2|\nabla (\pi \circ Q)|(\tilde{\Omega}\times(0,1)) \leq \liminf_{m\to \infty}2|\nabla (\pi \circ Q_{\e_m}) |(\tilde{\Omega}\times(0,1)).
\end{align}
By throwing away the absolutely continuous and Cantor parts of $|\nabla (\pi \circ Q)|$, cf. \cite[Equation 2.11]{ambrosio}, we can estimate the left hand side of the previous inequality to get
\begin{align*}
2\int_{J_Q}d_{\sqrt{W}}(Q^+(x),Q^-(x))\,d\mathcal{H}^2(x) \leq \liminf_{m\to \infty}2|\nabla (\pi \circ Q_{\e_m}) |(\tilde{\Omega}\times(0,1)).
\end{align*}
Here $J_Q\subset \tilde{\Omega}\times(0,1)$ is the jump set of $Q$ and $Q^+$, $Q^-$ are the trace values from opposite sides of $J_Q$, cf. \cite[Definition 1.3]{ambrosio}. Note that \eqref{bd1} implies that the $L^1$ limit $Q$ of $Q_{\e_m}$ cannot depend on $z$, so we may identify $Q$ with its canonical restriction to $\tilde{\Omega}$ and write
\begin{align}\label{abcd}
2\int_{J_Q}d_{\sqrt{W}}(Q^+(x),Q^-(x))\,d\mathcal{H}^1(x) \leq \liminf_{m\to \infty}2|\nabla (\pi \circ Q_{\e_m}) |(\tilde{\Omega}\times(0,1)).
\end{align}
Then $Q \in \mathcal{A}_0$ since $W(Q)=0$ a.e. and $Q$ is independent of $z$, hence
\begin{align*}\notag
F_0(Q) &= \sum_{i,j=1}^n d_{\sqrt{W}}(P_i,P_j)\mathcal{H}^1(\partial^* A_i \cap \partial^* A_j)+2\sum_{i=1}^n \int_{\pa A_i \cap \pa \Om} d_{\sqrt{W}}(P_i,g(x)) \,d\mathcal{H}^1(x),
\end{align*}
which is easily seen to be equal to
\begin{align*}
2\int_{J_Q}d_{\sqrt{W}}(Q^+(x),Q^-(x))\,d\mathcal{H}^1(x).
\end{align*}
Combining this with \eqref{abcd} and \eqref{i2} yields
\begin{equation}\notag
F_0(Q)=2\int_{J_Q}d_{\sqrt{W}}(Q^+(x),Q^-(x))\,d\mathcal{H}^1(x)\leq \liminf_{m\rightarrow \infty} \tilde{F}_{\e_m}(Q_{\e_m})+C(g)\operatorname{meas}(\tilde{\Omega}\setminus \Om).
\end{equation}
By choosing $\tilde{\Omega}$ so that $\tilde{\Om}\setminus \Om$ has small measure, we can make $C(g)\operatorname{meas}(\tilde{\Omega}\setminus \Om)$ arbitrarily small. The lower semicontinuity property \eqref{lsc} follows. The proof is now complete in the case where $g$ is as in \eqref{g2}; the proof for $g$ as in \eqref{g1} is similar.\color{black}
\end{proof}
\begin{proof}[Proof of the existence of a recovery sequence satisfying \eqref{reco1}, \eqref{reco2}] Fix any $Q_0 \in L^1$. If $\linebreak F_0(Q_0)=\infty$, then the constant sequence $Q_\e=Q_0$ satisfies \eqref{reco1} and \eqref{reco2}. Therefore we may assume $F_0(Q_0)\neq \infty$, and thus $Q_0 \in \mathcal{A}_0$. In particular, $Q_0$ does not depend on $z$. Our recovery sequence will also be independent of $z$; therefore in the rest of the proof, we will work on $\Omega$ rather than $\Omega \times (0,1)$. To obtain a recovery sequence for $Q_0$, we will approximate $Q_0$ by a sequence of functions $\{Q_h\}_{h\in \N}$ and furnish recovery sequences $\{ Q^h_\e\}_{\e>0}$ for each $h \in \N$.\par\color{black}
Before proceeding with the details, we summarize the main ingredients. In the construction of the recovery sequence for each $Q_h$, we will use the rotational symmetry of $W$ described at the beginning of this section. This symmetry allows us to create a family of geodesics under the degenerate metric $d_{\sqrt{W}}$ indexed by $x \in \pa \Om$ which are Lipschitz in $x$. These curves form the boundary layer near $\pa \Om$ which contributes the cost associated with the transition from the boundary data $g$ to the values of $Q_h$. After some modifications due to the fact that the boundary data $g$ is non--constant, we can utilize similar estimates as in \cite{baldo} to provide an upper bound for the $F_\e$ energy of $Q_\e^{h}$ coming from this boundary layer. Away from $\pa \Om$, we will use the construction from \cite[Section 4]{ambrosio}. A diagonal argument will then yield the desired result. The proof is divided into several steps.\par
$Step$ $1$. In this step, we construct the sequence $\{Q_h\}_{h\in \N}$. Applying \cref{poly} to the sets
$$A_i:=\{x \in \Om : Q_0(x) \in P_i\},$$
we obtain a sequence $\{\tilde{A}_1^h, \dots , \tilde{A}_n^h\}_{h \in \N}$ of collections of polygons which satisfy the conditions of \cref{poly}. We claim that we can modify the polygons $\tilde{A}_i^{h}$ and obtain polygonal partitions $\{A_1^{h},\dots,A_n^h\}$ for each $h$ such that for each $A_i^{h}$, in a neighborhood of $\pa \Om$,
\begin{equation}\label{perpp}
\pa A_i^{h}\textup{ is a union of line segments, each intersecting }\pa \Om\textup{ at a }90^\circ \textup{ angle;}
\end{equation}
and the $\{A_1^{h},\dots,A_n^h\}$ still satisfy the conditions of the lemma; see Fig. \ref{Fig. 1}. Indeed, for each $h$, this can be done by making the neighborhood of $\pa \Om$ on which we modify the polygons sufficiently small. By restricting to a further subsequence of the $A_i^h$, we can assume that
\begin{equation}\label{ae}
\mathbbm{1}_{A_i^h} \rightarrow \mathbbm{1}_{A_i} \textrm{ a.e. as } h \rightarrow \infty.
\end{equation}
Define
\[ Q_h = \begin{cases} 
      Q_0(x) & \textrm{if } x\in A^h_i \cap A_i \textup{ and }\dist(x,\pa \Om) \geq 1/h\\
      \alpha_i & \textrm{if } x\in A^h_i \setminus A_i \textup{ or }\dist(x,\pa \Om) < 1/h,
   \end{cases}
\]
where $\alpha_i$ is any fixed element of $P_i$ chosen independently of $h$. We have defined $Q_h$ to be locally constant near $\pa \Om$ to simplify the boundary layer construction there. Then \eqref{ae} implies that $Q_h \rightarrow Q_0$ a.e. on $\Om$ since for almost every fixed $x\in\Om$, $\mathbbm{1}_{A_i^h}(x)=\mathbbm{1}_{A_i}(x)$ for large enough $h$ depending on $x$. Applying the dominated convergence theorem yields
\begin{equation}\label{converge}
Q_h \rightarrow Q_0 \textrm{ in } L^1(\Omega;\s).
\end{equation}
In addition, due to \cref{poly}
\begin{equation}\label{dia}
F_0(Q_h) \rightarrow F_0(Q_0).
\end{equation}
Properties \eqref{converge} and \eqref{dia} will allow us to diagonalize recovery sequences for $Q_h$ to obtain a recovery sequence for $Q_0$.\par\color{black}
\afterpage{
\begin{figure}
\includegraphics{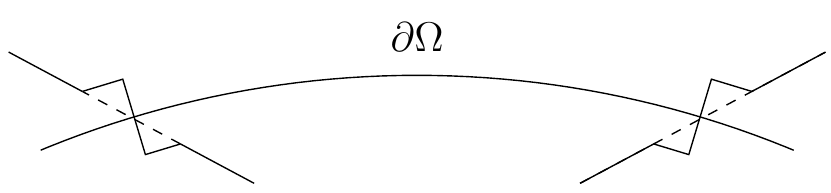}
\caption{The dashed lines represent the portions of $\pa \tilde{A}_i^{h}$ which are modified to obtain $\pa A_i^{h}$. The $A_i^{h}$ can be constructed so that the perimeter of $A_i^{h}$ is arbitrarily close to the perimeter of $ \tilde{A}_i^{h}$.}
 \label{Fig. 1}
\end{figure}
}
$Step$ $2$. We fix $h \in \N$ and present some preliminaries necessary for construction of the boundary layer near $\pa \Om$ for the recovery sequence of $Q_h$. First, in order to construct the boundary layer near $\pa \Om$ which bridges a constant matrix $\alpha_i$ to the boundary data $g$, we need to construct a family of paths indexed by $x\in \pa \Om$ connecting $\alpha_i$ to $g(x)$ which have sufficient smoothness in $x$ and are geodesics in the degenerate metric $d_{\sqrt{W}}$. In general, this might be quite difficult, but in our particular case, the choice of boundary data $g$ and the symmetry of $W$ will facilitate the process. We recall by \eqref{eq} that $d_{\sqrt{W}}(g(x),\alpha_i)$ does not depend on the choice of $x \in \pa \Om$. Therefore, for each connected component of $\pa \Om$, we fix a point $x_l$, $1 \leq l \leq L$, find geodesics connecting $\alpha_i$ to $g(x_l)$, and then suitably modify them to obtain geodesics connecting $\alpha_i$ to other $g(x)$.\par
For any $x_l\in \pa \Om$, and let $\gamma_{i,x_l}:[0,1/2]\rightarrow \mathcal S$ satisfy $\gamma_{i,x_l}(0)=g(x_l)$, $\gamma_{i,x_l}(1/2)=\alpha_i$ and
\begin{equation}\label{almost}
\int_0^{1/2} \sqrt{W(\gamma_{i,x_l}(t))}|\gamma_{i,x_l}'(t)| \,dt = d_{\sqrt{W}}(\alpha_i,g(x_l)).
\end{equation}
The existence of such a geodesic follows from \cite{zun}. We can without loss of generality assume that $|\gamma_{i,x_l}'|$ is constant. For each $x$ in the connected component of $\pa \Om$ containing $x_l$, we have $g(x) = r_{\theta(x)} g(x_l) r_{\theta(x)}^T$ for some $\theta(x)$ and $1 \leq l \leq L$, and we can choose the function $\theta:\pa \Om\to \R$ to be continuous on $\pa \Om \setminus x_l$. Near $x_l$, the value of $\theta(x)$ approaches 0 coming from one side and a possibly non-zero multiple of $2\pi$ from the other side. We rotate $\gamma_{i,x_l}$ so it connects $g(x)$ to $r_{\theta(x)} \alpha_i r_{\theta(x)}^T$ and then connect that point to $\alpha_i$ by letting $\theta(x)$ go to $0$. We define
\begin{equation}\label{gamma} \gamma_{i,x}(t) = \begin{cases} 
      r_{\theta(x)}\gamma_{i,x_l}(t)r_{\theta(x)}^T & \textrm{if } 0\leq t \leq 1/2\\
      r_{\theta(x)(-2t+2)} \alpha_i r_{\theta(x)(-2t+2)}^T & \textrm{if } 1/2< t \leq 1.
   \end{cases}
\end{equation}
It is straightforward to see that away from $x_l$, $\gamma_{i,x}(t)$ is locally Lipschitz in $x$ and $t$. In fact, viewing $\gamma_{i,x}(t)$ as a function of $t$ indexed by $x$, it can be quickly seen that the Lipschitz constants for each $\gamma_{i,x}(\cdot)$ are uniformly bounded in $x$. Note that for $t\geq 1/2$, $W(\gamma_{i,x}(t))=0$. We can use this fact and \eqref{yay} to calculate
\begin{align} \notag
\int_0^1 \sqrt{W(\gamma_{i,x}(t))}|\gamma_{i,x}'(t)| \,dt 
&= \int_0^{1/2}\sqrt{W(\gamma_{i,x_l}(t))}|\gamma_{i,x_l}'(t)| \,dt \\ \notag
&=d_{\sqrt{W}}(\alpha_i,g(x_l)) \\ \notag
&=d_{\sqrt{W}}(\alpha_i,g(x)) .
\end{align}
\par
We recall the following facts from \cite[Lemma 3.2]{baldo} which are useful in the construction. Consider the family of ordinary differential equations indexed by $x\in \pa \Om$ and $i=1,\dots , n$:
\begin{equation}\label{be3.2}
\left( \f{\partial}{\partial t} y_{\e,\delta}^{i,x}(t)\right)^2 = \f{\delta+ W(\gamma_{i,x}(y^{i,x}_{\e,\delta}(t)))}{\e^2|\gamma_{i,x}'(y_{\e,\delta}^{i,x}(t))|^2},
\end{equation}
where $0<\delta \ll 1$ is a fixed constant. We have for every $\e>0$, constants $C_{1,\delta}$, $C_{2,\delta}$, and $C_{3,\delta}$ independent of $\e$ and $x$, and strictly increasing solutions $y_{\e,\delta}^{i,x}:[0,C_{i,x}]\rightarrow[0,1]$ to \eqref{be3.2} such that
\begin{equation}\label{ii}C_{i,x} < C_{1,\delta}\e,\end{equation}
\begin{equation}\label{iii}y_{\e,\delta}^{i,x}(0)=0 \textup{ and }y_{\e,\delta}^{i,x}(C_{i,x})=1,\end{equation}
\begin{equation}\label{iv}\|(y_{\e,\delta}^{i,x})^{-1}\|_{L^\infty} \leq C_{2,\delta}\e,\end{equation}
\begin{align}\label{v}\textup{for fixed }t\textup{, }y_{\e,\delta}^{i,x}\textup{ is Lipschitz in }x\textup{ away from } x_0 \textup{ with }\\ \notag\textup{Lipschitz constant }C_{3,\delta}\textup{ independent of }t \textup{ as well }.\end{align}
The first three properties, \eqref{ii}--\eqref{v}, are all established in or follow quickly from \cite[Lemma 3.2]{baldo} along with the uniform Lipschitz bounds on $\gamma_{i,x}(\cdot)$. We remark that the fourth item, \eqref{v}, follows from applying Gronwall's inequality. It is convenient to have the $y_{\e,\delta}^{i,x}$ defined on one common interval, so let us extend each $y_{\e,\delta}^{i,x}$ to $[0,C_{1,\delta}\e]$ by setting $y_{\e,\delta}^{i,x}=1$ for $t>C_{i,x}$.\par\color{black}
$Step$ $3$. We will now present the construction of the recovery sequence $\{Q_\e^{h}\}_{\e>0}$ for $Q_h$ for fixed $h$. We recall from the first step that $Q_h$ is one of the approximations of the original element $Q_0$ of $\mathcal{A}_0$. We will in fact construct sequences $\{Q_\e^{h,\delta} \}$ for $0<\delta \ll 1$, prove that
\begin{equation}
\lim_{\delta\rightarrow 0} \limsup_{\e\rightarrow 0} F_\e(Q_\e^{h,\delta} )=F_0(Q_h),
\end{equation}
and then diagonalize over $\delta$ and $\e$. Let us fix $0 < \delta \ll 1$. By utilizing the result from \cite[Section 4]{ambrosio}, we obtain a sequence $\{\hat{Q}_\e^{h}\}$ such that
\begin{equation}\label{baldo1}
\hat{Q}_\e^{h} \rightarrow Q_h \textup{ in } L^1(\Om;\s) \textup{ as }\e \rightarrow 0
\end{equation}
and
\begin{equation}\label{plug}
\limsup_{\e\rightarrow 0} F_\e(\hat{Q}_\e^{h}) \leq \sum_{i,j=1}^n d_{\sqrt{W}}(P_i,P_j)\h(\pa^* A_i^h \cap \pa^*A_j^h)
\end{equation}
For now, we assume that there exists $C_{4}$ such that \begin{equation}\label{inn}\textup{if }x\in A_i^h\textup{ with }\dist(x,\pa A_i^h)\geq C_{4}\e\textup{ and }\dist(x,\pa \Om)\leq C_4,\textup{ then }\hat{Q}_\e^{h}(x)=\alpha_i.\end{equation} This assumption simplifies the calculations to follow, and we will see at the end of the proof that this assumption is not restrictive.
\par
The $\hat{Q}_\e^{h}$ do not assume the boundary values $g$, and the right hand side of \eqref{plug} does not account for cost associated to the boundary of $\Om$. To address these issues, we will modify $\hat{Q}_\e^{h}$ near $\pa \Om$; away from $\pa \Om$, the $\hat{Q}_\e^{h}$ will be unchanged. Briefly, we will set $Q_\e^{h,\delta}$ to be $\hat{Q}_\e^{h}$ on $\Om_{C_{1,\delta}\e}:=\{x\in \Om:\dist (x,\pa\Om)> C_{1,\delta}\e\}$ and then, using the curves $\gamma_{m,x}$, define a boundary layer which bridges the values of $\hat{Q}_\e^{h,\delta}$ to the boundary data $g(x)$ along segments normal to $\pa \Om$. Recall from \eqref{perpp} that each $\pa A_i^h$ intersects $\pa \Om$ at a $90^\circ$ angle; this fact allows us to avoid technicalities which would arise from a point $x\in A_i^h \cap \pa \Om_{C_{1,\delta}\e}$ whose projection $\sigma(x)$ onto $\pa \Om$ is in $A_j^h$ for $j\neq i$. We will need a different strategy near $x\in \pa \Om\cap\pa A_i^h \cap \pa A_j^h$ for $i\neq j$, due to the obvious fact that $\gamma_{i,x}$ do not vary smoothly in $i$. In addition, each $\gamma_{i,x}$ was not continuous at $x_l\in\pa \Om$, so we will need to cut out a small set near each $x_l$ and modify our analysis there as well.\par
For each $x\in \Om \setminus \Om_{C_{1,\delta}\e }$, consider its projection $\sigma(x)$ onto $\pa \Om$ and the inward pointing unit normal vector $\nu(\sigma(x))$ to $\pa\Om$ at $\sigma(x)$. By the condition \eqref{perpp} on $A_i^h$, we see that $\{\pa A_i^h \cap \pa A_j^h \cap \pa \Om: i\neq j  \}$ is finite; we enumerate this set as $\{x_l\}_{l=L+1}^{\tilde{L}}$. Let $\mathcal{C}_\e\subset \pa \Om$ be a finite union of curves $\mathcal{C}_\e^l$ contained in $\pa \Om$ for $1\leq l \leq \tilde{L}$, each of length $2C_{4}\e$ and centered around an $x_l$, so that $x_l$ divides $\mathcal{C}_\e^l$ into two pieces of length $C_{4}\e$. We define $$\Om_0^\e=\{x \in \Om \setminus \Om_{C_{1,\delta}\e} : \sigma(x) \in \mathcal{C}_\e \}.$$ 
Let us denote by $d(x)$ the distance from $x$ to $\pa \Om$ for $x \in \Om$. We now define
\[ Q_\e^{h,\delta}(x) = \begin{cases} 
      \hat{Q}_\e^{h}(x), & \textrm{if } x \in \Om_{C_{1,\delta}\e}, \\
      \gamma_{i,\sigma(x)}(y_{\e,\delta}^{i,\sigma(x)}(d(x))), & \textrm{if } x\in (\Om \setminus (\Om_{C_{1,\delta}\e}\cup \Om_0^\e).
   \end{cases}
\]
We note that for $x\in \pa \Om_{C_{1,\delta}}$ such that $x\in A_i^h$ but $\sigma(x) \in A_j^h$ for $i\neq j$, defining $Q_\e^{h,\delta}$ using $\gamma_{i,\sigma(x)}$ will result in a jump discontinuity. However, since $\pa A_j^h$ is normal to $\pa \Om$ in a neighborhood of $\pa \Om$ by \eqref{perpp} and $\hat{Q}_{\e}^{h}(x)=\alpha_i$ if $\dist(x,\pa A_i^h)\geq C_{4,\delta}\e$, for small enough $\e$, any such $x$ must be in $\Om_0^\e$. Hence for small $\e$, $Q_\e^{h,\delta}$ as defined thus far is Lipschitz. We have not yet defined $Q_\e^{h,\delta}$ on $\Om_0^\e$, but we will do so at the end of the proof.  \par\color{black}
First, we remark that $Q_\e^{h,\delta} \rightarrow Q_h$ in $L^1$ as $\e \rightarrow 0$ due to \eqref{baldo1} and the fact that $Q_\e^{h,\delta}$ is bounded on $\Om \setminus \Om_{C_{1,\delta}\e}$ independently of $\e$. Next, we estimate $F_\e(Q_\e^{h,\delta})$. Let us denote for any $U \subset \Om$ and tensor $Q$
$$
F_\e(Q,U)= \int_{U}\left(\e|\nabla Q|^2+\f{1}{\e}W(Q)\right)\,dx
.$$
Using \eqref{plug} yields
\begin{align}\notag
\limsup_{\e \rightarrow 0} F_\e(Q_\e^{h,\delta}, \Om_{C_{1,\delta}\e}) &=
\limsup_{\e \rightarrow 0} F_\e(\hat{Q}_\e^{h,\delta}, \Om_{C_{1,\delta}\e}) \\ \label{use}
&\leq\sum_{i,j=1}^n d_{\sqrt{W}}(P_i,P_j)\h(\pa A_i^h \cap \pa A_j^h\cap \Om).
\end{align}
Let us denote by $\Om_i^\e$ the set $\{x\in\Om \setminus (\Om_{C_{1,\delta}\e}\cup \Om_0^\e):x \in A_i^h  \}$. We now examine $F_\e(Q_\e^{h,\delta},\Omega_i^\e)$. We will make use of the map 
\begin{equation}\notag
P_t : \pa \Om \rightarrow \{x \in \Om : d(x)=t\}
\end{equation}
which sends $x \in \pa \Om$ to $x+t\nu(x)$, where $\nu(x)$ is the inward pointing normal to $\pa \Om$ at $x$. Because of the $C^2$ assumption on $\pa \Om$, $P_t$ is a $C^1$--diffeomorphism with Jacobian $J$ satisfying
\begin{equation}
\label{jac}
|J(P_t)(x)-1|\leq Ct
\end{equation}
for some $C>0$ independent of $x$ and $t$.\par
For each $x \in \Om_i^\e$, let $\tau=\tau(x)$ be a unit vector tangent to the level set of $d$ at $x$ and $\eta=\eta(x)$ be a unit vector perpendicular to $\tau$. We write 
\begin{align}
\notag F_\e(Q_\e^{h,\delta},\Om_{i}^\e) &= \int_{\Om_{i}^\e} \left(\e|\nabla Q_\e^{h,\delta}|^2 + \f{1}{\e}W(Q_\e^{h,\delta}) \right)\,dx\\ \notag
&= \int_{\Om_{i}^\e} \left(\e\left|\f{\pa}{\pa \tau}Q_\e^{h,\delta}\right|^2+\e\left|\f{\pa}{\pa \eta}Q_\e^{h,\delta}\right|^2 + \f{1}{\e}W(Q_\e^{h,\delta})\right)\,dx .\notag 
\end{align}
Now from \eqref{v}, it follows that
\begin{equation}\notag
\f{\pa}{\pa \tau}(Q_\e^{h,\delta}) = \f{\pa}{\pa \tau}(\gamma_{i,\sigma(x)}\circ y_{\e,\delta}^{i,\sigma(x)}\circ d)(x) 
\end{equation}
is bounded independently of $\e$ and $x$. Since $|\Om_i^\e|\leq C \e $, where $C$ is independent of $\e$, we have
\begin{equation}\notag
\int_{\Om_{i}^\e} \e\left|\f{\pa}{\pa \tau}Q_\e^{h,\delta}\right|^2 \,dx \leq \int_{\Om_{i}^\e} C\e \,dx \leq C\e^2.
\end{equation}
\color{black}Using now the coarea formula and the fact that $|\nabla d|=1$, we write
\begin{align}
\notag F_\e&(Q_\e^{h,\delta},\Om_{i}^\e)\\ \notag &\leq   \int_{\Om_{i}^\e} \left( \e\left|\f{\pa}{\pa \eta}Q_\e^{h,\delta}\right|^2 + \f{1}{\e}W(Q_\e^{h,\delta})  \right) \,dx +C\e^2\color{black} \\ \notag
&\leq \int_{\Om_{i}^\e} \left(\e\left|\f{\pa}{\pa \eta}(\gamma_{i,\sigma(x)}\circ y_{\e,\delta}^{i,\sigma(x)}\circ d)(x)\right|^2+\f{1}{\e}W(\gamma_{i,\sigma(x)}\circ y_{\e,\delta}^{i,\sigma(x)}\circ d)(x) \right) |\nabla d(x)| \,dx  \notag \\ \notag &\hspace{1cm} +C\e^2\color{black} \\ \notag
&=  \int_0^{C_{1,\delta}\e} \int_{\{d=t\}\cap \Om_i^\e} \left( \e (\gamma_{i,\sigma(x)}'(y_{\e,\delta}^{i,\sigma(x)}(t)))^2((y_{\e,\delta}^{i,\sigma(x)})')^2+ \f{1}{\e}W(\gamma_{i,\sigma(x)}\circ y_{\e,\delta}^{i,\sigma(x)})(t)\right) \,d\h(x) \,dt \\ \notag  &\hspace{1cm} +C\e^2\color{black}.
\end{align}
Then, recalling \eqref{be3.2} and using the map $P_t$, we have:
\begin{align}
\notag F_\e(Q_\e^{h,\delta},&\Om_{i}^\e) \\ \notag &\leq \int_0^{C_{1,\delta}\e}\int_{\{d=t\}\cap \Om_i^\e} 2|\gamma_{i,\sigma(x)}'(y_{\e,\delta}^{i,\sigma(x)}(t))||\delta+W(\gamma_{i,\sigma(x)}\circ y_{\e,\delta}^{i,\sigma(x)})(t)) |^{1/2} \,d\h(x) \,dt \\ \notag  &\hspace{1cm}+C\e^2\color{black} \\ \notag
&= \int_0^{C_{1,\delta}\e }\int_{\pa \Om \cap \Om_i^\e} 2|\gamma_{i,x}'(y_{\e,\delta}^{i,x}(t))||\delta +W(\gamma_{i,x}\circ y_{\e,\delta}^{i,x})(t)) |^{1/2} |J(P_t(x))|\,d\h(x) \,dt\\ \notag  &\hspace{1cm}+C\e^2\color{black}\\ \notag
&= \int_{\pa \Om \cap \Om^\e_i}\int_0^{C_{1,\delta}\e} 2|\gamma_{i,x}'(y_{\e,\delta}^{i,x}(t))||\delta+W(\gamma_{i,x}\circ y_{\e,\delta}^{i,x})(t)) |^{1/2} |J(P_t(x))| \,dt\,d\h(x)\\ \notag &\hspace{1cm} +C\e^2\color{black}.
\end{align}
Making the change of variables $s=y_{\e,\delta}^{i,x}(t)$ and recalling from the definition of $\Om_i^\e$ that $\pa \Om \cap \Om_i^\e \subset A_i$, we have
\begin{align}\notag
F_\e(Q_\e^{h,\delta},&\Om_{i}^\e)  \\ \notag &\leq\int_{\pa \Om \cap A_i}\int_0^{1} 2|\gamma_{i,x}'(s)||\delta+W(\gamma_{i,x}(s)) |^{1/2} |J(P_{(y_{\e,\delta}^{i,x})^{-1}(s)}(x))| \,ds \,d\h(x) +C\e^2\color{black}.
\end{align}
By using \eqref{jac} and the bound $\|(y_{\e,\delta}^{i,x})^{-1} \|_{L^\infty}\leq C_{2,\delta}\e$ from \eqref{iv}, we estimate $|J(P_{(y_{\e,\delta}^{i,x})^{-1}(s)})|\leq 1+C\e$. Taking the limsup on both sides as $\e\rightarrow 0$ gives
\begin{align}\notag
\ls F_\e(Q_\e^{h,\delta},\Om_{i}^\e)
&\leq \ls\int_{\pa \Om \cap B_i}\int_0^{1} 2|\gamma_{i,x}'(s)||\delta+W(\gamma_{i,x}(s)) |^{1/2}\,ds \,d\h(x)\\ \notag &\hspace{1cm} + C\e+C\e^2\color{black} \\ \label{refer} &\leq 2 \int_{\pa \Om \cap A_i^h} \int_0^{1} |\gamma_{i,x}'(s)||\delta+W(\gamma_{i,x}(s)) |^{1/2}\,ds \,d\h(x).
\end{align}\par
Using the right hand side of the previous inequality as our estimate for $\ls F_\e(Q_\e^{h,\delta},\Om_i^\e)$ yields, upon summing,
\begin{align}\notag
\ls \sum_{i=1}^n  F_\e(Q_\e^{h,\delta},\Om_{i}^\e)  &\leq 2 \sum_{i=1}^n \int_{\pa \Om\cap A_i^h} \int_0^{1} |\gamma_{i,x}'(s)||\delta+W(\gamma_{i,x}(s)) |^{1/2}\,ds \,d\h(x) \\ \notag
&= 2\sum_{i=1}^n \int_{ \pa \Om \cap A_i^h} d_{\sqrt{W}}(g(x),\alpha_i) \,d\mathcal{H}^1(x)+C\sqrt{\delta}
 \\ \label{part1}
&=2\sum_{i=1}^n \int_{ \pa \Om\cap A_i^h }d_{\sqrt{W}}(g(x),P_i) \,d\mathcal{H}^1(x)+C\sqrt{\delta},
\end{align}
where $C$ only depends on the the lengths of the $\gamma_{i,x}$, which in turn depend on $h$. Combining \eqref{use} and \eqref{part1}  yields
\begin{equation}\label{inside}
 \ls   F_\e(Q_\e^{h,\delta},\cup_{i=1}^n \Om_i^\e \cup \Om_{C_{1,\delta}\e}) \leq F_0(Q_h)+C_h\sqrt{\delta},
\end{equation}
We will take the infimum over $\delta$ at the end of the proof, so as not to confuse the order in which $\delta$ and $\e$ are sent to zero in the remaining estimates.\par
To finish proving the estimate
\begin{equation}\notag
\ls F_\e(Q_\e^{h,\delta},\Om) \leq F_0(Q_h)+C_h\sqrt{\delta},
\end{equation}
it suffices to define $Q_\e^{h,\delta}$ on $\Om_0^\e$ and show that
\begin{equation}\notag
F_\e(Q_\e^{h,\delta},\Om_0^\e)\rightarrow 0
\end{equation}
as $\e \to 0$ for any fixed $\delta>0$. We will define $Q_\e^{h,\delta}$ on this set so that the integrand $\e|\nabla Q_\e^{h,\delta}|^2+\f{1}{\e}W(Q_\e^{h,\delta})$ is $O(1/\e)$ there, and then show that the measure of this set is $O(\e^2)$. Since $\delta$ is fixed in this argument, we will suppress its appearance in the constants for the following estimates. \color{black}We assume that $Q_\e^{h,\delta}$ restricted $\overline{\Omega}\setminus (\Om_{C_{4}}\cup \Om_0^\e))$ satisfies:
\begin{equation}\notag
\begin{cases} 
      \|Q_\e^{h,\delta}\|_{L^\infty(\overline{\Om} \setminus(\Om_{C_{4}}\cup \Om_0^\e))}\leq C \\ \|\nabla Q_\e^{h,\delta} \|_{L^\infty(\overline{\Omega}\setminus (\Om_{C_{4}}\cup \Om_0^\e))}\leq \f{C}{\e},\\
   \end{cases}
\end{equation}
where $C$ is independent of $\e$. If $Q_\e^{h,\delta}$ did not satisfy these estimates, then $Q_\e^{h,\delta}$ can be modified inside $\Om_{C_4}$ near $\pa \Om_{C_4}$ and extended to $\Om \setminus \Om_{C_4}$ using level sets of the distance function so as to meet these requirements. This can be done by utilizing the techniques of \cite[Lemma 3.2]{fb}, in which it is shown that the boundary values of a sequence such as $Q_\e^{h,\delta}$ can be changed without increasing the total energy in the limit. The application of \cite[Lemma 3.2]{fb} also allows us to assume that \eqref{inn} holds as well. \par\color{black}
Let us consider $\{x\in \Om_0^\e: \sigma(x)\in \mathcal{C}_\e^l\}\subset\Om_0^\e$. Since $\{x\in \Om_0^\e: \sigma(x)\in \mathcal{C}_\e^l\}$ is a ``strip" of length $2C_{4,\delta}\e$ and width $C_{1,\delta}\e$, it is easy to see from the coarea formula that for fixed $\delta$,
\begin{equation}\notag
|\Om_0^\e| = O(\e^2).
\end{equation}
Now, since on the boundary of this strip, $\|\nabla Q_\e^{h,\delta} \|\leq \f{C}{\e}$, it can be quickly seen that $Q_\e^{h,\delta}$ can be extended to a Lipschitz function satisfying\begin{equation}\notag
\begin{cases} 
      \|Q_\e^{h,\delta}\|_{L^\infty(\{x\in \Om_0^\e: \sigma(x)\in \mathcal{C}_\e^l\})}\leq C \\ \|\nabla Q_\e^{h,\delta} \|_{L^\infty(\{x\in \Om_0^\e: \sigma(x)\in \mathcal{C}_\e^l\})}\leq \f{C_{\kappa}}{\e},\\
   \end{cases}
\end{equation}
where $C_\kappa\geq C$ is a constant depending on the curvature $\kappa$ of $\pa \Om$. If $\pa \Om$ is flat, so that $\Om_0^\e$ is a rectangle, then linearly interpolating the values of $Q_\e^{h,\delta}$ from the boundary of the rectangle along diagonal segments across the rectangle works and gives $C_\kappa=C$. We estimate
\begin{equation}\label{woo}
\int_{\{x\in \Om_0^\e: \sigma(x)\in \mathcal{C}_\e^l\}} \left(\e|\nabla Q_\e^{h,\delta}|^2+\f{W(Q_\e^{h,\delta})}{\e}\right)\,dx \leq \f{C_\kappa}{\e}|\Om_0^\e|=O(\e)
.\end{equation}
Since $\Om_0^\e$ is the union of the sets $\{x\in \Om_0^\e: \sigma(x)\in \mathcal{C}_\e^l\}$,  we have
\begin{equation}\label{woo1}
 \int_{\Om_0^\e} \left(\e|\nabla Q_\e^{h,\delta}|^2+\f{W(Q_\e^{h,\delta})}{\e}\right)\,dx =O(\e).
\end{equation}\par
Summing the estimates \eqref{part1}, \eqref{woo}, and \eqref{woo1} gives
\begin{equation}\notag
\ls F_\e(Q_\e^{h,\delta}) \leq  F_0(Q_h)+C_h\sqrt{\delta},
\end{equation}
so that \begin{equation}\label{why}\lim_{\delta\rightarrow 0}\ls F_\e(Q_\e^{h,\delta})\leq F_0(Q_h).\end{equation} Now diagonalizing over $\delta$ and $\e$, we obtain a recovery sequence $\{Q_\e^h\}_{\e>0}$ for $Q_h$.\color{black}\par
\textit{Conclusion of proof:} Combining \eqref{dia} and \eqref{why}, we have $$\lim_{h \to \infty} \lim_{\delta \to 0}\limsup_{\e \to 0} F_\e(Q_\e^{h,\delta})\leq \lim_{h\rightarrow \infty} F_0(Q_h)=F_0(Q_0) .$$ Having already diagonalized $\{Q_\e^{h,\delta}\}_{\e>0}$ over $\delta$ and $\e$ to obtain a recovery sequence $\{Q_\e^{h}\}_{\e>0}$ for $Q_h$, we diagonalize the recovery sequences $\{Q_\e^{h}\}_{\e>0}$ over $\e$ and $h$ and obtain a recovery sequence $\{Q_\e\}_{\e>0}$ for $Q_0$. The case for constant boundary data $g$ is simpler and follows from the above calculations.
\end{proof}
\section{Local Minimizers of $F_\e$}\label{local}
Now that we have proven $\Gamma$--convergence, we aim to prove the existence of local minimizers of $F_\e$. A similar theorem was proved for the Allen-Cahn functionals in \cite{ks89} by minimizing the functionals in an $L^1$--ball around an isolated $L^1$--local minimizer of the $\Gamma$--limit. The proof required the existence of such an isolated local minimizer in addition to $L^1$--compactness for any sequence of functions with bounded energies. Neither of these conditions holds in the problem we are considering. Indeed, for any local minimizer $Q_0$ of $F_0$, we have that $F_0(\rt^T Q_0 \rt)=F_0(Q_0)$. This follows from the rotational symmetry of $W$ described in the beginning of Section \ref{gproof}. Also, since the zero set $P$ of $W$ is higher--dimensional, as opposed to a finite collection of points, we cannot obtain $L^1$--compactness of a sequence of minimizers $Q_\e$. To account for both of these issues, we introduce the distance
$$
\Lambda(Q_1,Q_2) = \sum_i \|\p_i(Q_1)-\p_i(Q_2) \|_{L^1(\Om)}
$$
for any $Q_1$ and $Q_2$ in $\mathcal{A}_0$, cf. \eqref{pii} and \eqref{a00}, respectively. We observe that $\Lambda(Q_1,Q_2)=0$ if and only if for each $i$, $\{x:Q_1(x) \in P_i \}=\{x:Q_2(x) \in P_i \}$ up to a set of measure zero.\par\color{black}
\begin{prop}\label{comapctness}
Let $Q_\e$ be a sequence of maps from $\Omega \times(0,1)$ to $\s$ and assume that the sequence of energies $F_\e(Q_\e)$ is uniformly bounded. Then there exists a subsequence $\{Q_{\e_j}\}$ and $Q \in L^1(\Om \times (0,1);P)$ such that $\Lambda(Q_{\e_j},Q)\to 0$.
\end{prop}
\begin{remark}
We thank the anonymous referee for bringing to our attention a similar result stated without proof in \cite[Proposition 4.1]{ambrosio}. For the convenience of the reader we include the elementary argument.
\end{remark}
\begin{proof}
By similar arguments as in the proof of lower--semicontinuity in Section \ref{gproof}, we can truncate the $Q_\e$ to obtain $\tilde{Q}_\e$ such that $\Lambda(Q_\e,\tilde{Q}_\e)\to 0$ and $\tilde{Q}_\e$ are uniformly bounded in $L^\infty$. Hence if we obtain $Q$ such that $\Lambda(\tilde{Q}_{\e_m},Q)\to 0$, then $\Lambda(Q_{\e_m},Q)\to 0$ as well. Suppressing the tildes, from the $L^\infty$ bound we see that
$$
\| \p_i \circ Q_\e\|_{L^1(\Om)} \leq C < \infty.
$$
Next, using the calculations preceding \eqref{fe}, which replaced the surface integrals of $f_s$ by volume integrals, we have
\begin{align}\notag
 2\int_{\Om\times (0,1)} \sqrt{W(Q_\e)}|\nabla Q_\e|\,dx\,dz &\leq \int_{\Om\times (0,1)} \left(  \e|\nabla Q_\e|^2+\f{W(Q_\e)}{\e} \right) \,dx\,dz\\ \notag &\leq F_\e (Q_\e)+O(\e) \\ \notag &\leq C < \infty.
\end{align}
It is straightforward then to see that
\begin{equation}\notag
2\int_{\Om \times(0,1)}|\nabla(\p_i\circ Q_\e)|\,dx \,dz  \leq  2\int_{\Om\times (0,1)} \sqrt{W(Q_\e)}|\nabla Q_\e|\,dx\,dz\leq C < \infty
.\end{equation}
Thus $\{\p_i\circ Q_\e\}$ are uniformly bounded in $BV$.\par
It remains to construct a limiting element $Q$. For each $i$, up to a subsequence,
$$
\p_i(Q_{\e_j}) \rightarrow \omega_i \textup{ as }j \to \infty
$$
in $L^1(\Om\times (0,1))$ for some function $\omega_i$. We claim that the sets $E_i:=\{(x,z)\in \Om \times(0,1):\omega_i(x,z)=0\}$ partition $\Omega \times (0,1)$ up to sets of measure zero. To see this, first suppose by way of contradiction that there exists a set $A \subset \Omega \times (0,1)$ of positive measure such that none of the $\omega_i$'s are zero on $A$. By restricting to a further subsequence, we can assume for each $i$ that $\p_i(Q_{\e_j})$ converges almost everywhere and hence, by Egoroff's Theorem, almost uniformly on $A$. Since $\omega_i>0$ on $A$ for each $i$, we can obtain a set $B \subset A$ of positive measure and an $\eta>0$ such that on $B$, $$\p_i(Q_{\e_j})>\eta>0$$ for each $i$ and for $j$ sufficiently large. But if the distance $\p_i$ to $P_i$ is greater than $\eta$ on $B$ for each $i$, it follows that for $j$ sufficiently large, $W(Q_{\e_j})>\tilde{\eta}$  on $B$ for some $\tilde{\eta}>0$. We also recall from \eqref{fe} that for a sequence such as $\{Q_\e\}$ with bounded energies we can replace $F_\e(Q_\e)$ by $\tilde{F}_\e(Q_\e)$ up to an error of order $O(\e)$. Combining these observations we deduce that
$$
\f{|B|\tilde{\eta}}{\e_j} \leq \int_B \f{W(Q_{\e_j})}{\e_j}\,dx\,dz \leq \tilde{F}_{\e_j}(Q_{\e_j}) \leq C < \infty
,$$
which is a contradiction for $j$ sufficiently large. We conclude that the union of the $E_i$'s contains $\Om$ up to a set of measure zero. To see that the $E_i$ are disjoint, note that if $\p_i(Q_{\e_j})$ is very close to zero, then for $k \neq i$, $\p_k(Q_{\e_j})$ must be away from zero. Therefore, the sets $\{E_i\}$ have empty intersection, so they partition $\Om$ up to a set of measure zero.\par
We see that each $\omega_i$ has finite range, since it can only take the values $d_{\sqrt{W}}(P_i,P_k)$. Using the partition $\{E_i\}$, we can take any $Q\in L^1(\Om \times (0,1);\s)$ such that $\{ (x,z)\in \Om \times (0,1) : Q(x,z) \in P_i \}=E_i$ as our limiting element. As a specific example of such a $Q$, define $Q \equiv \alpha_i \in P_i$ on each $E_i$, where $\alpha_i$ is any constant in $P_i$. 
\end{proof}
We now prove the existence of local minimizer for $F_\e$ if there exists a $\Lambda$--isolated local minimizer of $F_0$.
\begin{theorem}\label{mini}
Let $Q_0$ be an isolated $\Lambda$--local minimizer in the sense that there exists $\delta > 0$ such that
$$
F_0(Q_0) < F_0(Q)
$$
if $0<\Lambda(Q,Q_0) \leq \delta$. Then there exists $\e_0>0$ and a family $\{Q_\e \}_{\e<\e_0}$ such that
\begin{equation}\label{min1}
Q_\e \text{ is a }\Lambda \text{--local minimizer of } F_\e
\end{equation}
and
\begin{equation}\label{min2}
\Lambda(Q_0,Q_\e) \rightarrow 0
.\end{equation}
\end{theorem}
\begin{remark}\label{mike}
We note that $Q_\e$ are in fact $H^1$--local minimizers of $F_\e$. This is because the distance $\Lambda$ is weaker than the $H^1$ metric.
\end{remark}
\begin{proof}
The direct method in the calculus of variations yields for each $\e>0$ a $Q_\e$ which minimizes $F_\e$ on the ball $\{Q:\Lambda(Q,Q_0)\leq \delta\}$. By the existence of recovery sequences for $F_\e$ established in \cref{t1}, we know that for $\e$ sufficiently small, the recovery sequence $\{Q_\e^0\}$ for $Q_0$ is contained in $B$. We then have
\begin{equation}\label{bleh}
\liminf_{\e \rightarrow 0} F_\e(Q_\e) \leq \liminf_{\e \rightarrow 0} F_\e(Q_\e^0) = F_0(Q_0).
\end{equation}\par
We will now show that for sufficiently small $\e$, $\Lambda(Q_\e,Q_0)<\delta$, thus establishing that the $Q_\e$ are local minimizers of $F_\e$. Suppose this does not hold, so that there exists a subsequence $\{ \e_j\}$ such that $\Lambda(Q_{\e_j},Q_0)=\delta$. By \cref{comapctness}, we obtain $Q\in L^1(\Om \times (0,1);P)$ such that up to a further subsequence, still denoted by $Q_{\e_j}$,
$$
\Lambda(Q_{\e_j},Q)\to 0.
$$
We will show that $F_0(Q)\leq F_0(Q_0)$ and $Q\in \{Q:\Lambda(Q,Q_0)\leq \delta\}$. This will contradict that $Q_0$ is a $\Lambda$--isolated local minimizer of $F_0$.\par
We observe $\|d_{\sqrt{W}}(Q_{\e_j},Q) \|_{L^1}\leq \| \p_i \circ Q_{\e_j}\|_{L^1}+\| \p_i \circ Q\|_{L^1}$, so that $\| d_{\sqrt{W}}(Q_{\e_j},Q)\|\to 0$. Next, examining the proof of the lower semicontinuity condition \eqref{lsc}, we that see the $L^1$ convergence of $ d_{\sqrt{W}}(Q_{\e_j},Q)$ to $0$ along with the fact that $Q$ takes values in $P$ are in fact sufficient conditions to conclude that
\begin{equation}\notag
F_0(Q) \leq \liminf_{j \to \infty} F_{\e_j}(Q_{\e_j}).
\end{equation} But we also have from \eqref{bleh} that
\begin{equation}\notag
\liminf_{j \rightarrow \infty}  F_{\e_j}(Q_{\e_j}) \leq F_0(Q_0).
\end{equation}
 It follows that $F_0(Q)\leq F_0(Q_0)$. Combining this with our assumption that $\Lambda(Q_0,Q)=\delta$ we obtain a contradiction to the fact that $Q_0$ is an isolated $\Lambda$--local minimizer of $F_0$. We have thus shown that $\Lambda(Q_0,Q_\e)<\delta$ for sufficiently small $\e$. The proof that $\Lambda(Q_\e,Q_0)\rightarrow 0$ proceeds using similar reasoning.
\end{proof}\color{black}
We point out that the arguments in the preceding theorem apply in more general scenarios. Below, we formulate one such generalization.\par
Let $\mathcal{N}\subset \R^l$ be an open, bounded domain with smooth boundary. Also, we fix a smooth, disconnected, and bounded set $\mathcal{M}\in\R^k$ with components $\mathcal{M}_1, \cdots, \mathcal{M}_n$. Let $W : \R^k \to [0,\infty)$ satisfy $W^{-1}({0})=\mathcal{M}$ and $W(v)\to \infty$ as $|v|\to \infty$. Consider the functionals
\begin{equation}\notag
H_\e(u) = \int_{\mathcal{N}}\left(\e |\nabla u|^2 + \f{1}{\e}W(u)\right)\,dx,
\end{equation}
for maps $u$ satisfying either a prescribed volume constraint or a Dirichlet boundary condition $g$. Let $\p_i$ be the distance to $\mathcal{M}_i$ under the degenerate Riemannian metric with the conformal factor $\sqrt{W}$. If we define the distance
\begin{equation}\label{L}
\Lambda_{\mathcal{M},\mathcal{N}}(u_1,u_2)=\sum_{i=1}^n \|\p_i(u_1)-\p_i(u_2) \|_{L^1(\mathcal{N})},
\end{equation}
then any sequence $\{u_\e\}$ with bounded energy is $\Lambda_{\mathcal{M},\mathcal{N}}$--compact. Suppose that in the $L^1(\mathcal{N};\mathcal{M})$ topology, the $H_\e$ $\Gamma$--converge to a functional $H_0$ as $\e \to 0$. Then we have the following theorem.
\begin{theorem}\label{general}
Let $u_0$ be an isolated $\Lambda_{\mathcal{M},\mathcal{N}}$--local minimizer in the sense that there exists $\delta > 0$ such that
$$
H_0(u_0) < H_0(u)
$$
if $0<\Lambda_{\mathcal{M},\mathcal{N}}(u,u_0) \leq \delta$. Then there exists $\e_0>0$ and a family $\{u_\e \}_{\e<\e_0}$ such that
\begin{equation}\label{genmin1}
u_\e \text{ is a }\Lambda_{\mathcal{M},\mathcal{N}} \text{--local minimizer of } H_\e
\end{equation}
and
\begin{equation}\label{genmin2}
\Lambda_{\mathcal{M},\mathcal{N}}(u_0,u_\e) \rightarrow 0
.\end{equation}
\end{theorem}
\begin{remark}
Using the techniques from the proof of \cref{t1}, the $\Gamma$--convergence of $H_\e$ could be established in the case of a volume constraint. In the case of a Dirichlet boundary condition, one would need to be able to construct a smoothly varying family of paths connecting the boundary data $g$ to any element in the domain of a map in the admissible set for $H_0$. Provided this is possible, the $\Gamma$--convergence could be proved using our techniques.
\end{remark}
\section{Isolated Local Minimizers for $F_0$}\label{iso}
In this section, we will prove the existence of an isolated local minimizer to a partitioning problem on certain two--dimensional domains. For a domain $\Om$, we will refer to a partition of $\Om$ as an ordered pair $(C,D)$ of subsets of $\Om$ with finite perimeter and disjoint measure theoretic interiors such that 
$$
|\Om - (C \cup D)|=0
.$$
The notation $| \cdot |$ refers to the Lebesgue measure. We introduce the notation
$$
\theta(C,x)=\lim_{r\rightarrow 0} \f{|C\cap B(x,r)|}{|B(x,r)|}
$$
to refer to the density of a set $C$ at a point $x$. If the above limit does not exist, we will use $\underline{\theta}(C,x)$ and $\overline{\theta}(C,x)$ to refer to the corresponding lim inf and lim sup, respectively. We define the measure theoretic interior of a set $C$ as the set of points $x$ where $\theta(C,x)=1$ and denote it by $C^i$. The measure theoretic boundary of a finite perimeter set $C$ is defined as the set
$$
\pam C = \{x:0<\overline{\theta}(C,x)\}\cap \{x:\underline{\theta}(C,x)<1\}
.$$
Note that finite perimeter sets are only defined up to sets of Lebesgue measure zero. Throughout, in order to avoid ambiguity, we will always use the measure--theoretic closure of a set $C$, the set of points where $\overline{\theta}(C,x)>0$, as its representative, and we will denote the measure--theoretic boundary of $C$ by $\pa C$ rather than $\pam C$.\par
The functional $F_0$ is defined on partitions $(C,D)$ of $\Om$ by the formula
\begin{equation}\label{f0}
F_0(C,D) = c_1\h(\pa C \cap \pa \Omega)+c_2\h(\pa D \cap \pa \Omega)+c_3\h(\pa C \cap \pa D)
.\end{equation}
We reuse the notation $F_0$ since $F_0$ as in \eqref{f0} is a special case of the $F_0$ defined in \eqref{f0intro} when there are only two connected components of the zero set $P$ of $W$.\par
Let us assume the following inequalities for the constants $c_i$:
$$0<c_1 < c_2$$
and
$$c_2<c_1+c_3.$$
The first inequality is natural, since in the case that $c_1=c_2$, the cost on $\pa \Om$ is the same for every competitor and hence minimizing $F_0$ reduces to minimizing the interfacial length of any partition. In this case, the problem is the same as that studied in \cite{ks89}. Phrased in the language of an Allen--Cahn type problems with a potential $W$, the second inequality is an assumption that the triangle inequality is strict for the degenerate Riemannian metric with conformal factor $\sqrt{W}$.\par
We now describe the type of domains for which we will be able to construct a local minimizer to $F_0$; see Fig. \ref{Fig. 2}. Let $\Omega \subset \R^2$ be a simply connected $C^2$ domain, and denote the outward unit normal vector to $\Omega$ by $\nu_\Om$. Assume that $\Om$ contains a line segment $\overline{PQ}$ such that $P,Q \in \pa \Om$ and there exists a unit vector $v$ perpendicular to $P-Q$ satisfying the following conditions:
\begin{equation}\label{contact}
v \cdot \nu_\Om = \f{c_1-c_2}{c_3}<0\textup{ at } P \textup{ and }Q,
\end{equation}
and in neighborhoods $N_1$, $N_2$ of $P$ and $Q$, respectively,
\begin{equation}\label{convex}
\pa \Om\cap N_i \textup{ is the graph of a strictly convex function over an interval of length }l.
\end{equation}
The first condition is the contact angle condition which arises as a necessary condition for criticality. It represents a local balance between interfacial and boundary energy.\par
Define $\Gamma_1$ to be $\pa \Om \cap N_1$ and similarly $\Gamma_2=\pa \Om \cap N_2$. Note that under our assumptions, $v \cdot \nu_\Om$ is a monotone function of the arc--length variable along $\Gamma_1$ and $\Gamma_2$. We can now define the candidate $(A,B)$ for the isolated local minimizer of $F_0$. Let $\pa A \cap \pa B = \overline{PQ}$ and choose $A$ such that for each $i$, $v \cdot \nu_\Om$ is strictly larger on $\Gamma_i \cap \pa A$ than on $\Gamma_i \cap \pa B$; see Fig. \ref{Fig. 2}. The inequality $c_1<c_2$ forces us to specify $A$ and $B$ in this manner, since $F_0$ is not symmetric regarding the boundary costs. Consider the following variational problem:
\begin{equation}\label{var}
\textrm{minimize }F_0(C,D)\textrm{ over partitions }(C,D) \textrm{ satisfying } |A \triangle C|+|B\triangle D| \leq\delta.
\end{equation}
We will prove the following theorem:
\begin{theorem}\label{main} There exists $\delta=\delta(\Om,\{c_i\})$ such that $(A,B)$ is  the unique solution to the problem \eqref{var}. 
\end{theorem}
\afterpage{
\begin{figure}
\includegraphics{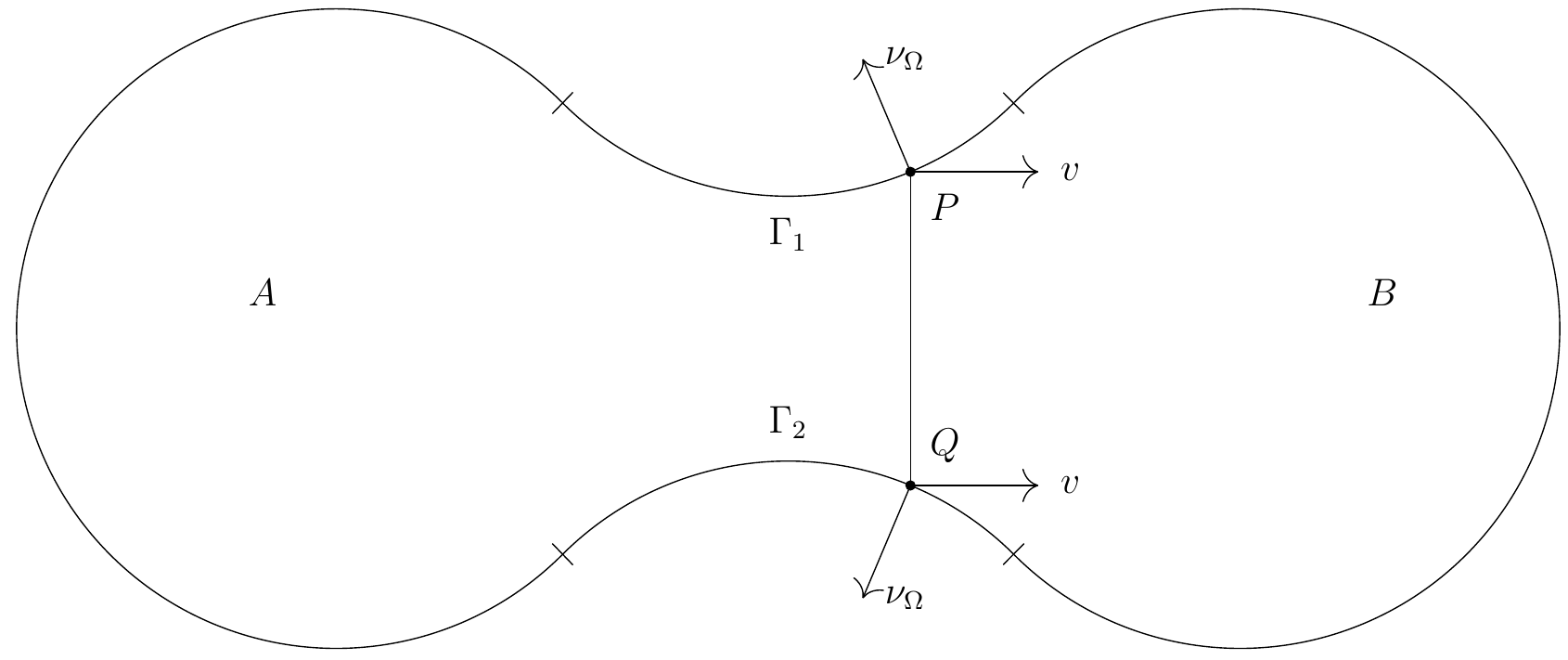}
\caption{A dumbbell domain with an isolated local minimizer $(A,B)$ of $F_0$.}
 \label{Fig. 2}
\end{figure}
}
Note that the direct method and compactness in BV yield that a minimizer exists for the variational problem \eqref{var}. To prove \cref{main}, we will first show in \cref{corral2} that the boundary $\pa A' \cap \pa B' \cap \Om$ of any solution $(A',B')$ to \eqref{var} must be ``uniformly close" to $\overline{PQ}$. Then, using a calibration argument, we will show that $(A,B)$ is an isolated local minimizer.\par
We make a few remarks about $\Om$ which will be used throughout the proof of \cref{main}. Denote the set $\{x \in \Om : d(x,\pa \Om)<\eta\}$ by $\Om_ \eta$. Let $\gamma:[0,L)\rightarrow \pa \Om$ parametrize $\pa \Om$ by arclength. The assumption that $\pa \Om$ is $C^2$ implies that for some small $\eta_0$, the map 
$$
T:[0,L) \times[0,\eta_0]\rightarrow \overline{\Om}_{\eta_0}
$$
defined by
$$
T(s,y):=\gamma(s)-y\nu_\Om(\gamma(s))
$$
is a local $C^1$--diffeomorphism. Under this map, the distance from a point $T(s,y)\in \Om$ to $\pa \Om$ is $y$. We will refer to the nearest point projection of a point $x\in \Om$ onto $\pa \Om$ by $\sigma(x)$ and the projection $(s,y)\rightarrow(s,0)$ onto the $s$--axis in $\R^2$ by $\pi$. Note that with these conventions, $$T \circ \pi= \sigma \circ T.$$\par
The following lemma provides a lower bound on the interior perimeter of a finite perimeter set $V \subset \Om_ \eta$ in terms of the length of its projection $\sigma(V)$ onto $\pa \Om$ and will be needed in the proof of \cref{corral2}.
\begin{lemma}\label{dist}
Let $V$ be a set of finite perimeter such that $V \subset T([a,b] \times [0,\eta))$ for some $\eta<\eta_0$. Then if we set $\kappa_\infty:= \max\{ |\gamma_{ss}| \}$ (the maximum of the curvature of $\pa \Om$) we have the estimates
\begin{equation} \label{close1}
(1-\eta \kappa_\infty)\h(\sigma(V)) \leq \h( \pa V \cap T((a,b) \times (0,\eta)))
\end{equation}
and
\begin{equation}\label{close2}
\h(\pa V \cap \pa \Om) \leq \h(\sigma(V)).
\end{equation} 
\end{lemma}
\begin{remark}Regarding \cref{dist}, we remark that sets of finite perimeter which differ by a set of Lebesgue measure zero (and are thus equivalent) could have projections onto $\partial \Om$ which differ by a set of possible large $\h$ measure. This would invalidate \eqref{close1}. Recall however that  we employ the convention for a set $V$ of finite perimeter, $x \in V$ if and only if $\overline{\theta}(V,x)>0$. It is this representative of $V$ for which the lemma holds.\end{remark}
\begin{proof}
Let us first consider the case where $\pa V \cap T((a,b) \times (0,\eta))$ is a single smooth curve. Let $(s(t),y(t)):(c,d)\rightarrow (a,b) \times (0,\eta)$ be a smooth curve parametrized by arc--length such that $T \circ (s,y):(c,d) \rightarrow \Omega$ parametrizes $\pa V \cap T((a,b) \times (0,\eta))$. Using the chain rule on $T\circ (s,y)$ and the identities $\langle \gamma_s, \nu_\Om\rangle=0$, $\langle \nu_\Om,(\nu_\Om)_s\rangle=0$, and $1=|\gamma_s|^2=|\nu_\Om|^2$, we write
\begin{align}
\notag \h( \pa V \cap T((a,b) \times (0,\eta)))&=\int_c^d \langle (T \circ (x,y))',(T \circ (x,y))'\rangle^{1/2} \,dt \notag \\  \notag
&=\int_c^d \langle\gamma_s s'+y'\nu_\Om+y (\nu_\Om)_ss',\gamma_s s'+y'\nu_\Om+y (\nu_\Om)_ss'\rangle^{1/2} \,dt \notag \\
&= \notag \int_c^d \left(s'^2+2s'^2y\langle \gamma_s,(\nu_\Om)_s \rangle+y'^2+y^2s'^2| (\nu_\Om)_s|^2\right)^{1/2} \,dt.
\end{align}
Recall that by construction, $s'^2+y'^2=1$, which also implies $s'^2\geq s'^4$. Also note that $|\langle \gamma_s,(\nu_\Om)_s \rangle| = |(\nu_\Om)_s|$, $|(\nu_\Om)_s |\leq \kappa_\infty$, and $|y|\leq \eta$. Continuing the previous line using these observations yields the estimate
\begin{align}
\notag \h( \pa V \cap T((a,b) \times (0,\eta)))	&\geq \int_c^d  (1-2|s'^2y||(\nu_\Om)_s |+s'^4y^2|(\nu_\Om)_s |^2)^{1/2}\,dt \\ \notag
&= \int_c^d (1-|s'^2y||(\nu_\Om)_s |)\,dt \\ \notag 
&\geq \int_c^d 1-\eta\kappa_\infty\,dt \\ \label{case} 
&\geq (1-\eta\kappa_\infty)\h(\pi(T^{-1}(V))).
\end{align}
Next, since $\gamma$ parametrizes the boundary $\pa \Om$ by arc--length, we see that $\restr{T^{-1}}{\pa \Omega}$  preserves $\h$ measure. In addition, $T^{-1} \circ \sigma= \pi \circ T^{-1}$. These two facts imply that $$\h(\pi(T^{-1}(V)))=\h(T^{-1}(\sigma(V)))=\h(\sigma(V)).$$
This equality in conjunction with \eqref{case} yields \eqref{close1}, namely
\begin{equation}\notag
(1-\eta\kappa_\infty)\h(\sigma(V)) \leq \h( \pa V \cap T((a,b) \times (0,\eta)))
\end{equation}
for the case where $\pa V \cap T((a,b) \times (0,\eta_0))$ is a single smooth curve. For $\pa V \cap T((a,b) \times (0,\eta))$ still smooth but with more than one component, we apply the above calculation to each component and sum the results to obtain \eqref{close1}. Finally, for such $V$, \eqref{close2} is immediate.\par
Now, for arbitrary $V$ with finite perimeter, by mollifying $V$ and using super--level sets of the mollifications as in Chapter 1 of \cite{giusti}, we obtain a sequence of smooth sets $V_n$ approximating $V$ and take limits in \eqref{close1} for $V_n$.
\end{proof}
We conclude the preliminaries with a standard result which we will need in proof of \cref{main}. The proof is found in \cite[p. 1065]{ziemer}.
\begin{lemma}\label{plane}
Let $E\subset \R$ be a set of finite perimeter. Suppose there exists a point $x_0 $ in $\pa E$ with the property that for some closed cube $Q_0$ centered at $x_0$, $\nu_E(x)$ is a constant $v$ for all $x \in Q_0 \cap \pa E$. Then $Q_0\cap \pa E=Q_0\cap P$, where $P$ is the hyperplane containing $x_0$ with normal $v$.
\end{lemma}

\begin{lemma}\label{corral2}
Let $\Om\subset \R^2$ be a domain satisfying the conditions \eqref{contact} and \eqref{convex}. Then there exists $\delta=\delta(\Om)>0$ such that
\begin{equation}
\pa A' \cap \pa B' \subset \{x:d(x,\overline{PQ}) < 10\sqrt{\delta}\}\cup ( \sigma^{-1}(\Gamma_1\cup\Gamma_2)\cap\Om_{\dd})\notag
\end{equation}
for any solution $(A',B')$ of \labelcref{var}.
\end{lemma}

\begin{proof}[Proof of \cref{corral2}] For the convenience of the reader, we first summarize the main arguments. Let $(A',B')$ be a solution of \eqref{var}. The proof is divided into two steps.\par
In the first step, we argue that $\pa A' \cap \pa B'$ must be contained in a neighborhood of $\pa B$. The proof of this step follows closely the proof in \cite{ziemer}. The main idea is that since $(A',B')$ is close to $(A,B)$ in $L^1$, it will incur significant cost $c_3(\pa A' \cap \pa B')$ if it protrudes too far from $\pa B$. Next, using the result of the first step, we are able to further restrict $\pa A' \cap \pa B'$ to a neighborhood of $\overline{PQ}\cup\Gamma_1 \cup \Gamma_2$. Since $\pa A'\cap \pa B'$ is contained in a neighborhood of $\pa B$, the result of \cref{dist} implies that near $\pa \Om\cap \pa B$, the interior perimeter $\h(\pa A' \cap \pa B')$ is bounded below by the boundary perimeter $\h(\pa A' \cap \pa \Om)$. Together with the inequality $c_1+c_3>c_2$, this will allow us to enlarge $B'$ into $B''$ near $\pa \Om$ and exchange greater cost associated with
$$
c_1\h(\pa A' \cap \pa \Om)+c_3\h(\pa A' \cap \pa B')
$$
for lesser cost associated with
$$c_2\h(\pa B'' \cap \pa \Om).$$\par
We begin with some preliminaries regarding the parameters in the proof. The parameter $\delta$ in the statement of the lemma depends on the constants $\{c_i\}$ and the domain $\Om$. Let $\gamma:[p,q]\rightarrow \pa \Om$ parametrize $\pa B \cap \pa \Om$ with $\gamma(p)=P$ and $\gamma(q)=Q$, and choose $p<b_1<b_2<q$ such that $\gamma([p,b_1])\subset \Gamma_1$ and $\gamma([b_2,q])\subset \Gamma_2$. Fix any $a_1\in(p,b_1)$ and then choose $a_2 \in (b_2,q)$ such that $b_1-a_1=a_2-b_2$ (this is merely for convenience later). We choose $\delta$ small enough to satisfy two conditions:
\begin{equation}\label{>1}
T([a_1,a_2]\times\{\dd\})\subset \{ x : d(x, \pa B) = \dd \}
\end{equation}
and
\begin{equation}\label{>2}
0<\f{20\sqrt{\delta}}{c_1+c_3(1-\dd\kappa_\infty)-c_2}<b_1-a_1.
\end{equation}
Recall our assumption that $c_1+c_3>c_2$, which implies that any sufficiently small $\delta$ satisfies \eqref{>2}. Also, the first condition might not hold for $\delta$ too large because a point in $B$ at distance $\dd$ from $\pa \Om$ might be within $\dd$ of $\overline{PQ}$ and hence not in  $\{ x : d(x, \pa B) = \dd \}$; see Fig. \ref{Fig. 3}. We now proceed with the proof.\par
\afterpage{
\begin{figure}
\includegraphics{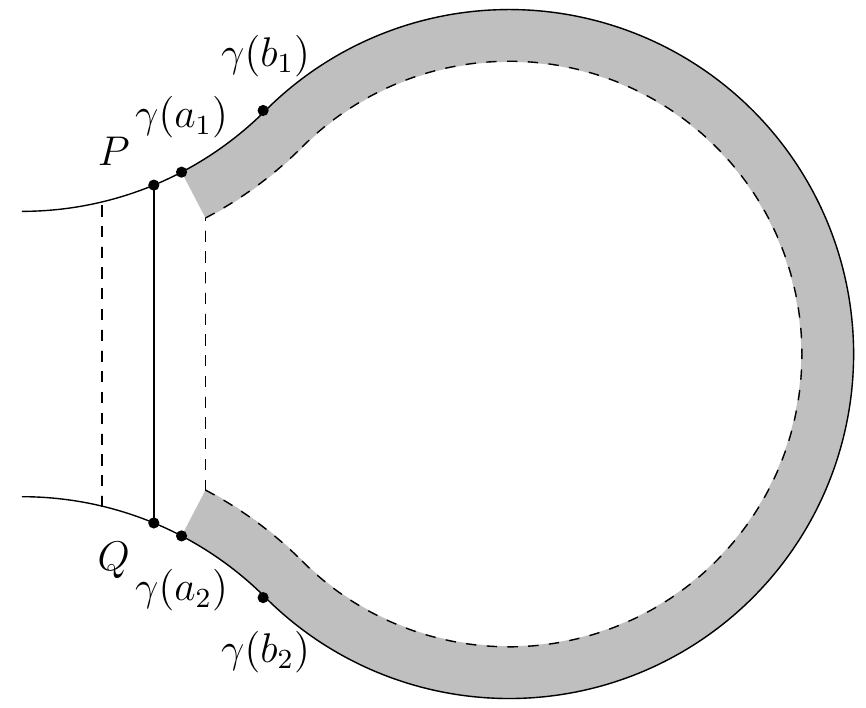}
\caption{The shaded region is $T([a_1,a_2]\times[0,\dd])$. The dashed curves comprise $\{x\in \Om:d(x,\pa B)=\dd\}$. Note that $T([a_1,a_2]\times\{\dd\})$ is entirely contained in this level set. }
 \label{Fig. 3}
\end{figure}
}
$Step$ $1$. We argue that	
\begin{equation}\label{corral1}
\pa A' \cap \pa B' \subset \{x\in \Om:\textup{dist}(x,\pa B) < 10\sqrt{\delta}\}
;\end{equation}
see Fig. \ref{Fig. 4}. To prove this, one follows closely the method of proof of \cite[Lemma 3.1]{ziemer} with minor changes. For convenience, we summarize the argument here. First, we argue that $\{ x \in A: \textup{dist}(x,\pa B) \geq \dd\}\cap B'=\emptyset$. Consider slices $S_t=\{ x \in A:\textup{dist}(x,\pa B)=t\}$. If there exists a $t_0<10\sqrt{\delta}$ such that $\h(S_{t_0} \cap B')=0$, then setting
$$
A''=A' \cup \{ x \in A: \textup{dist}(x,\pa B) \geq t_0\}
$$ 
yields a modified partition $(A'',B'')$ which still satisfies the $L^1$ constraint in \eqref{var}. Assume for contradiction $\{ x \in A: \textup{dist}(x,\pa B) > t_0\}\cap B'\neq\emptyset$. Then, by following the calculations from \cite[p. 1067]{ziemer}, we see that $(A'',B'')$ satisfies
$$
c_3\h(\pa A'' \cap \pa B'') < c_3 \h(\pa A' \cap \pa B').
$$
Note that since $A' \subset A''$, $B''\subset B'$, and $c_1<c_2$,
$$
c_1\h(\pa A'' \cap \pa \Om)+c_2\h(\pa B'' \cap \pa \Om)\leq c_1\h(\pa A' \cap \pa \Om)+c_2\h(\pa B' \cap \pa \Om),
$$
so in fact $F_0(A'',B'')<F_0(A',B')$, which contradicts the fact that $(A',B')$ is a minimizer of \eqref{var}. We conclude that $\{ x \in A: \textup{dist}(x,\pa B) > t_0\}\cap B'= \emptyset$, which together with the definition of $t_0$ yields the desired result. On the other hand, suppose there does not exist $t_0<10\sqrt{\delta}$ such that $\h(S_{t_0} \cap B')=0$. The $L^1$ constraint from \eqref{var} and the mean value theorem yield $T_0\in(\sqrt{\delta},2\sqrt{\delta})$ such that $\h(S_{T_0}\cap B)<\sqrt{\delta}$. It can then be shown (using the calculations from \cite[p. 1068]{ziemer}) that setting
$$
A''=A' \cup \{ x \in A: \textup{dist}(x,\pa B) \geq T_0\}
$$
results in a partition $(A'',B'')$ such that $F_0(A'',B'')<F_0(A',B')$, which again contradicts the minimality of $(A',B')$. This is due to the fact that $T_0$ was chosen so that the interior perimeter gain of at most $\sqrt{\delta}$ from $c_3\h(\pa A'' \cap S_{T_0})$ is offset by the loss of perimeter of $c_3\h(\pa A' \cap \{x\in A : 2\sqrt{\delta}<\textup{dist}(x,\pa B)\})$. In either case, we obtain that $\{ x \in A: \textup{dist}(x,\pa B) \geq \dd\}\cap B'=\emptyset$. 
The corresponding result, that $\{ x \in B: \textup{dist}(x,\pa B) \geq \dd\}\cap A'=\emptyset$, is proved similarly, by setting
$$
B'' = B' \cap \{x\in B : \textup{dist}(x,\pa B)\geq t\}
$$
for some carefully chosen $t$. These two results then yield \eqref{corral1}, since we see that
\begin{align}\notag
\pa A' \cap \pa B' &\subset
\Om \setminus \left( \{x \in A : \textup{dist}(x,\pa B) \geq \dd\} \cup \{x \in B : \textup{dist}(x,\pa B) \geq \dd\}\right) \\ \notag
&= \{x\in \Om:\textup{dist}(x,\pa B) < 10\sqrt{\delta}\}.
\end{align}
Observe that in Step $1$, we have avoided modifying $(A',B')$ by enlarging $B'$ near $\pa \Om$. The reason for this is that a priori, the loss of interior perimeter resulting from such a modification might be offset by an increase in the cost on $\pa \Om$ due to the inequality $c_1<c_2$. This difficulty will be addressed in the subsequent step.
\par
\afterpage{
\begin{figure}
\includegraphics{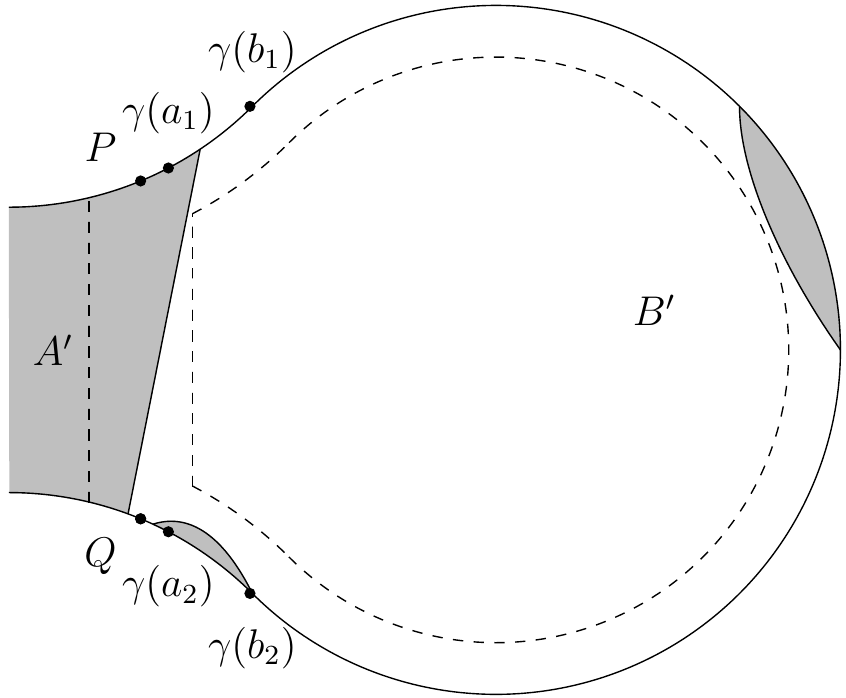}
\caption{The shaded set is $A'$ and the unshaded set is $B'$. The boundary $\pa A' \cap \pa B'$ is contained in a $\dd$ neighborhood of $\partial B$ due to Step $1$.}
\label{Fig. 4}
\end{figure}
}
$Step$ $2$. The second step, similar to the first, has two main parts. First, we eliminate ``islands" of $A'$ in $T((a_1,a_2)\times(0,\dd))$ separated from the main portion of $A'$ by slices of the form $T(\{d\}\times(0,\dd))$. Consider the set $\{d \in (a_1,a_2) : \h(T(\{d\}\times(0,\dd)) \cap A')=0\}$, and suppose it is non--empty. We choose $d_1$ and $d_2$ from this set with $d_1\leq d_2$; if possible we choose $d_i$ such that $\gamma(d_i) \subset \Gamma_i$. See Fig. 4 for an example of the case where $d_1$ can be chosen so that $\gamma(d_1)\in \Gamma_1$ but $d_2$ cannot be chosen so that $\gamma(d_2)\in \Gamma_2$. Now define 
$$
B''=B' \cup T([d_1,d_2]\times[0,\dd])
$$
and $A''$ as the measure--theoretic closure of $\Om \setminus B''$. Note that $(A'',B'')$ still satisfies the $L^1$ constraint from \eqref{var}. We claim that 
\begin{equation}\notag
B'' \triangle B', A'' \triangle A' \subset T([d_1,d_2]\times[0,\dd)).
\end{equation}
This will simplify the calculation below of $F_0(A',B')-F_0(A'',B'')$ since it will allow us to only examine $T([d_1,d_2]\times[0,\dd))$. It is straightforward to see that $B'' \triangle B'$ and $A'' \triangle A'$ are contained in $T([d_1,d_2]\times[0,\dd])$, so to prove the claim it remains to show that $B'' \triangle B'$ and $A'' \triangle A'$ have empty intersection with $T([d_1,d_2]\times\{\dd\})$. This holds because of the choice of $\delta$ in \eqref{>1}. By the result of Step $1$ and the inclusion \eqref{>1}, $T([d_1,d_2]\times\{\dd\})\subset (B')^i$, which implies it has empty intersection with $B'' \triangle B'$ and $A'' \triangle A'$. Therefore, we can calculate the difference in energies $F_0(A',B')-F_0(A'',B'')$ as 
\begin{align}\label{five}
F_0(A',B')-F_0(A'',B'')&=c_1\h(\gamma(d_1,d_2)\cap\pa A')+c_2\h(\gamma(d_1,d_2)\cap\pa B') \\ \notag &\hspace{.6cm}+c_3\h(T([d_1,d_2]\times(0,\dd))\cap \pa A')-c_2\h(\gamma(d_1,d_2)).
\end{align}
Assuming that $|A'\cap T([d_1,d_2]\times[0,\dd])|>0$, we will show that $F_0(A',B')-F_0(A'',B'')>0$, which contradicts the minimality of $(A',B')$. First, we bound the first three terms on the right hand side of \eqref{five} from below. As a preliminary estimate, note that \cref{dist} with $V=T((d_1,d_2)\times(0,\dd))\cap A'$ implies that
\begin{equation}\notag
\h(\gamma(d_1,d_2)\cap \pa A')\leq \h(\gamma(d_1,d_2)\cap\sigma(A')).\end{equation}
Using this in conjunction with the inequality $c_1-c_2<0$ we write for the first two terms of \eqref{five}
\begin{align}\notag
c_1\h(\gamma(d_1,d_2)\cap\pa A')+c_2\h(&\gamma(d_1,d_2)\cap\pa B') \\ \notag
&= (c_1-c_2)\h(\gamma(d_1,d_2)\cap \pa A')+ c_2\h(\gamma(d_1,d_2)) \\ \notag
&\geq (c_1-c_2)\h(\gamma(d_1,d_2)\cap\sigma(A')) + c_2\h(\gamma(d_1,d_2)) \\ \label{one}
&= c_1\h(\gamma(d_1,d_2)\cap\sigma(A'))+c_2\h(\gamma(d_1,d_2)\setminus\sigma(A')).
\end{align}
For the third term of \eqref{five}, also apply \cref{dist} to obtain
\begin{equation}\label{two}
c_3\h(T([d_1,d_2]\times(0,\dd))\cap \pa A') \geq c_3(1-\dd\kappa_\infty) \h(\gamma(d_1,d_2)\cap \sigma(A')).
\end{equation}
Then adding \eqref{one} and \eqref{two} yields
\begin{align}\label{arg}
c_1\h(\gamma(d_1,d_2)&\cap\pa A')+c_2\h(\gamma(d_1,d_2)\cap\pa B')+c_3\h(T([d_1,d_2]\times(0,\dd))\cap \pa A') \\ \notag
&\geq c_1\h(\gamma(d_1,d_2)\cap\sigma(A'))+c_2\h(\gamma(d_1,d_2)\setminus\sigma(A'))\\ \notag &\hspace{.6cm}+c_3(1-\dd\kappa_\infty)\h(\gamma(d_1,d_2)\cap\sigma(A')).
\end{align}
Now observe that \eqref{>2} implies that $c_1+c_3(1-\dd\kappa_\infty)>c_2$ and that the assumption $|A'\cap T([d_1,d_2]\times[0,\dd])|>0$ implies $\h(\gamma(d_1,d_2)\cap\sigma(A'))>0$. Together these observations imply that
\begin{equation}\notag
c_1\h(\gamma(d_1,d_2)\cap\sigma(A'))+c_3(1-\dd\kappa_\infty)\h(\gamma(d_1,d_2)\cap\sigma(A'))>c_2\h(\gamma(d_1,d_2)\cap\sigma(A')).
\end{equation}
Using this inequality for the first and third terms on the right hand side of \eqref{arg}, we obtain our estimate for the first three terms of \eqref{five}:
\begin{align}\notag
c_1\h(\gamma(d_1,d_2)&\cap\pa A')+c_2\h(\gamma(d_1,d_2)\cap\pa B')+c_3\h(T([d_1,d_2]\times(0,\dd))\cap \pa A') \\ \notag
&> c_2\h(\gamma(d_1,d_2)\cap\sigma(A'))+c_2\h(\gamma(d_1,d_2)\setminus\sigma(A')) \\ \label{three}
&=c_2\h(\gamma(d_1,d_2)).
\end{align}
Combining \eqref{five} and \eqref{three} yields
\begin{equation}\notag
F_0(A',B')-F_0(A'',B'')>c_2\h(\gamma(d_1,d_2))-c_2\h(\gamma(d_1,d_2)) =0.
\end{equation}
But this contradicts the fact that $(A',B')$ is a solution to \eqref{var}, so it must be the case that $|A'\cap T([d_1,d_2]\times[0,\dd])|=0$. It follows that $\pa A' \cap T((d_1,d_2)\times[0,\dd])=\emptyset$, which together with the result of Step 1 gives
\begin{equation}
\notag \pa A' \cap \pa B' \subset \{x\in\Om:\textup{dist}(x,\pa B) < 10\sqrt{\delta}\} \setminus T((d_1,d_2)\times[0,\dd])
\end{equation}
If it was possible to choose $\gamma_i(d_i) \in \Gamma_i$ for $i=1,2$, the lemma is proved, since
\begin{align}\label{contain}
\{x\in\Om:\textup{dist}(x,\pa B) < &10\sqrt{\delta}\} \setminus T((d_1,d_2)\times[0,\dd]) \\ \notag&\subset \{x:\textup{dist}(x,\overline{PQ}) < 10\sqrt{\delta}\}\cup ( \sigma^{-1}(\Gamma_1\cup\Gamma_2)\cap\Om_{\dd}).
\end{align}
Then it remains consider the scenario in which $\{d \in (a_1,a_2) : \h(T(\{d\}\times(0,\dd)) \cap A')=0\}=\emptyset$ or $d_i$ cannot be chosen such that $\gamma(d_i) \in \Gamma_i$. Fix $i=1$; $i=2$ is handled similarly. Define
\begin{equation}\notag
B''=B' \cup T([a_1,b_1]\times[0,\dd])
\end{equation}
and $A''$ as the measure theoretic closure of $\Omega \setminus B''$. Note that $(A'',B'')$ still satisfies the $L^1$ constraint. As in the previous part of Step 2, we can calculate
\begin{align}\label{five1}
F_0(A',B')-F_0(A'',B'')&=c_1\h(\gamma(a_1,b_1)\cap\pa A')+c_2\h(\gamma(a_1,b_1)\cap\pa B') \\ \notag &\hspace{.6cm}+c_3\h(T([a_1,b_1]\times(0,\dd))\cap \pa A')-c_2\h(\gamma(a_1,b_1)) \\ \notag &\hspace{.6cm}-c_3\h(T(\{a_1,b_1\}\times(0,\dd))\cap \pa A').
\end{align}
First, we estimate the fifth term from below:
\begin{equation}\label{zero1}
-c_3\h(T(\{a_1,b_1\}\times(0,\dd))\cap \pa A') \geq -20\sqrt{\delta}
.\end{equation}
Using the same reasoning as preceding \eqref{arg}, we can estimate the first three terms of \eqref{five1}:
\begin{align}\notag
c_1\h(\gamma(a_1,b_1)&\cap\pa A')+c_2\h(\gamma(a_1,b_1)\cap\pa B')+c_3\h(T([a_1,b_1]\times(0,\dd))\cap \pa A') \\ \notag
&\geq c_1\h(\gamma(a_1,b_1)\cap\sigma(A'))+c_2\h(\gamma(a_1,b_1)\setminus\sigma(A'))\\ \notag &\hspace{.6cm}+c_3(1-\dd\kappa_\infty)\h(\gamma(a_1,b_1)\cap\sigma(A')).
\end{align}
The assumption that $d_1$ either doesn't exist or can't be chosen such that $\gamma(d_1)\subset \Gamma_1$ implies that $\gamma(a_1,b_1)\cap\sigma(A')=\gamma(a_1,b_1)$. After using this in the previous inequality and combining with \eqref{zero1}, we obtain
\begin{align}\label{uhoh}
F_0(A',B')-F_0(A'',B'')
&\geq c_1\h(\gamma(a_1,b_1))+c_3(1-\dd\kappa_\infty)\h(\gamma(a_1,b_1))-20\sqrt{\delta} \\ \notag
&= c_1(b_1-a_1)+c_3(1-\dd\kappa_\infty)(b_1-a_1)-20\sqrt{\delta}.
\end{align}
Recalling \eqref{>2}, we rearrange it as
$$
 c_1(b_1-a_1)+c_3(1-\dd\kappa_\infty)(b_1-a_1)-20\sqrt{\delta}>0
.$$
But it immediately follows from this inequality and \eqref{uhoh} that $F_0(A',B')-F_0(A'',B'')>0$, which contradicts the minimality of $(A',B')$. It must then be the case that $\{d \in (a_1,a_2) : \h(T(\{d\}\times(0,\dd)) \cap A')=0\}\neq\emptyset$ and $\gamma(d_1) \in \Gamma_1$. It can be argued similarly that $\gamma(d_2)\in \Gamma_2$. In view of \eqref{contain}, the lemma is proven.
\end{proof}

Finally we can prove \cref{main}.
\begin{proof}[Proof of \cref{main}] Fix $\delta$ small enough to satisfy the assumptions of \cref{corral2}. Then the direct method yields the existence of a solution $(A',B')$ to \eqref{var}. By \cref{corral2}, we have that
$$
\pa A' \cap \pa B' \subset \{x:\textup{dist}(x,\overline{PQ}) < 10\sqrt{\delta}\}\cup ( \sigma^{-1}(\Gamma_1\cup\Gamma_2)\cap\Om_{\dd}).
$$
In particular, this implies that $(A',B')$ satisfies
\begin{equation}\label{cant}
\pa A' \cap \pa \Omega \subset \Gamma_1 \cup \Gamma_2 \cup (\pa \Om \cap \pa A),\textrm{ and } \pa B' \cap \pa \Omega \subset \Gamma_1  \cup \Gamma_2 \cup (\pa \Om \cap \pa B).
\end{equation}
We can now use a calibration argument to show that $(A,B)$ is the only minimal partition satisfying \eqref{cant} and hence the unique minimizer of \eqref{var}. Denote by $\nu_\Om$ the measure theoretic exterior normal to $\Om$. Let $(C,D)$ be an admissible partition. By Cauchy--Schwarz and the Gauss--Green theorem, we have
\begin{align}
F_0(C,D)\notag &= c_1\h(\pa C \cap \pa \Omega)+c_2\h(\pa D \cap \pa \Omega) + c_3\h(\pa C \cap  \pa D) \\ \label{sharp}
&\geq c_1 \h(\pa C \cap \pa \Omega)+c_2\h(\pa D \cap \pa \Omega) + c_3 \int_{\pa C \cap  \pa D} (v \cdot \nu_C )\,d\h \\
&= c_1 \h(\pa C \cap \pa \Omega)+c_2\h(\pa D \cap \pa \Omega)  -c_3 \int_{\pa C \cap \pa \Omega} (v \cdot \nu_\Om)\,d\h \notag \\
&=  \int_{\pa C \cap (\Gamma_1 \cup \Gamma_2)}(c_1 -c_3v \cdot \nu_\Om)\,d\h +\int_{(\pa C \cap \pa \Omega)\setminus (\Gamma_1 \cup \Gamma_2)}(c_1 -c_3v \cdot \nu_\Om)\,d\h\notag \\ &\hspace{.6cm}+\int_{\pa D \cap (\Gamma_1 \cup \Gamma_2)}c_2\,d\h +\int_{(\pa D \cap \pa \Omega)\setminus (\Gamma_1 \cup \Gamma_2)}c_2\,d\h.  \notag
\end{align}
Note that because of the restrictions in \eqref{cant} on the admissible partitions, the second and fourth integrals in the previous lines do not depend on $(C,D)$. In fact, $(\pa C \cap \pa \Omega)\setminus (\Gamma_1 \cup \Gamma_2) = (\pa A \cap \pa \Omega)\setminus (\Gamma_1 \cup \Gamma_2)$ and $(\pa D \cap \pa \Omega)\setminus (\Gamma_1 \cup \Gamma_2) = (\pa B \cap \pa \Omega)\setminus (\Gamma_1 \cup \Gamma_2)$. Thus if we set
$$
I:=\int_{(\pa C \cap \pa \Omega)\setminus (\Gamma_1 \cup \Gamma_2)}(c_1 -c_3v \cdot \nu_\Om)\,d\h+\int_{(\pa D \cap \pa \Omega)\setminus (\Gamma_1 \cup \Gamma_2)}c_2\,d\h,
$$
the above inequality becomes
\begin{equation}\label{only}
F_0(C,D) \geq I + \int_{\pa C \cap  (\Gamma_1 \cup \Gamma_2)}(c_1 -c_3v \cdot \nu_\Om)\,d\h+\int_{\pa D \cap  (\Gamma_1 \cup \Gamma_2)}c_2\,d\h,
\end{equation}
where $I$ does not depend on the choice of partition. Now, due to the strict convexity of $\Gamma_1$ and $\Gamma_2$ and the assumption that $c_1 -c_3v \cdot \nu_\Om=c_2$ at $P$ and $Q$, the following inequalities hold:
\begin{equation}\notag
c_1 -c_3v \cdot \nu_\Om<c_2 \textrm{ on }\pa A \cap \Omega
\end{equation}
and
\begin{equation}\notag
c_1 -c_3v \cdot \nu_\Om>c_2 \textrm{ on }\pa B \cap \Omega.
\end{equation}
Consequently, the only partitions which minimize the right hand side of \eqref{only} are those satisfying
\begin{equation}\label{same}
\pa C \cap \pa \Om = \pa A \cap \pa \Om \textrm{ and } \pa D \cap \pa \Om = \pa B \cap \pa \Om.
\end{equation}
In addition, the inequality in \eqref{sharp} is sharp unless \begin{equation}\label{same1}
v = \nu_C\ \hspace{.5cm}\h \textup{ a.e. on } \pa C \cap  \Om.
\end{equation}
Thus by applying \eqref{plane}, it follows that if \eqref{sharp} is an equality, then $\pa C \cap \pa D$ is a line segment. We have shown that $(A,B)$ is the only partition satisfying \eqref{same} and \eqref{same1} and therefore the unique minimizer of $F_0$ over all partitions satisfying \eqref{cant}, and so the theorem follows.
\end{proof}
\section{Asymptotic Energy of Minimizers of $F_\e$ for Boundary Data with Degree}\label{degree}
In this section we consider the behavior of minimizers of $F_\e$ with the surface term
\begin{equation}\label{name}
f_s = \gamma|(\I - \z \otimes \z )Q\z|^2
\end{equation}
and also subject to the boundary condition given by \eqref{g2}, $g(x)=-3\beta \left( n(x) \otimes n(x) - \f{1}{3} \I \right) $. The invariance of $Q$--tensors under reflection through the origin implies that \begin{equation}\label{degreee} \textup{the degree of }g \textup{ is an integer multiple of one--half};\end{equation}see \cite[Definition 2]{bpp}. We will denote this integer by $k$. Throughout this section, we will assume that the temperature is low enough so that the uniaxial nematic state with nematic order parameter $s_\star$ is the minimizer of $f_{LdG}$. From the description \eqref{state} of the zero set of $\F$, we see that the only effect of perturbing $\F$ by $2f_s$ as in \eqref{name} is to require now that $\z$ be an eigenvector of any minimizer of $\F+2f_s$. Hence the zero set $P$ of $W=f_{LdG} + 2f_s$ is the union of the circle $P_1:=\left\lbrace s_\star\left(m \otimes m - \f{1}{3}\I\right):m\in \mathbb{S}^1\times \{0\}\right\rbrace$ and the point $P_2:= s_\star\left(\z \otimes \z - \f{1}{3}\I\right)$. Let us assume $\beta$ is such that \begin{equation}\label{assu}\p_1\left(g \right)<\p_2\left(g\right) ,\end{equation} i.e., the boundary data $g$ is closer to the circle $P_1$ than to the point $P_2$ in the degenerate metric $d_{\sqrt{W}}$. We recall the notation $P_0$ for the image of $g$. In the following theorem regarding the energy of the minimizers $Q_\e$ of $F_\e$, we capture both the leading order boundary layer contribution reminiscent of vector Allen--Cahn and the lower order vortex contribution reminiscent of Ginzburg--Landau.
\begin{theorem}\label{asy} Let $\Omega$ be a simply--connected domain and $g$ have degree $k$. Assume that $f_s$ is given by $  \gamma|(\I - \z \otimes \z )Q\z|^2$. Then the minimizers $Q_\e$ of $F_\e$ satisfy the asymptotic development
\begin{equation}\label{t6}
F_\e(Q_\e) = 2\int_{\partial \Omega} \p_1(g(x)) \, d\h(x) + \e s_\star^2\pi k \log\f{1}{\e} + O(\e)
\end{equation}
as $\e\rightarrow 0$.
\end{theorem}
The proof of the theorem will consist of showing that the right hand side of \eqref{t6} bounds $F_\e(Q_\e)$ from above and from below.\par
\begin{proof}[Proof of the upper bound.] We will first prove the upper bound by constructing a sequence of functions $R_\e$ which satisfy
\begin{align}
F_\e(R_\e) \leq 2 \int_{\partial \Omega} \p_1(g(x)) \, d\h(x) + \e s_\star^2\pi k \log\f{1}{\e} + O(\e).
\end{align}
The $R_\e$ will be independent of $z$, so we will refer to them as functions of $x\in\Omega$ and treat the surface term as a bulk term, since for $R_\e=R_\e(x)$ only we have
\begin{align}
F_\e(R_\e) &= \int_{\Omega \times (0,1)} \left( \e|\nabla_{x}R_\e|^2+\f{f_{LdG}(R_\e)}{\e}\right) \,dx\,dz + \f{1}{\e} \int_{\Omega \times \{0,1\}}f_s(R_\e) \,dx \\ \notag
&= \int_{\Omega }\left( \e|\nabla_{x}R_\e|^2+ \f{W(R_\e)}{\e}\right) \,dx .
\end{align}
Our strategy will be to combine the construction of \cite{bpp} away from $\pa \Om$ with a boundary layer near $\pa \Om$. The interior construction will contribute the $\e \log \f{1}{\e}$ term and the boundary layer will contribute the order 1 term in \eqref{t6}. Throughout the estimates, the generic constant $C$ varies from line to line but does not depend on $\e$.\par
We first present some preliminaries regarding $W$ and $d_{\sqrt{W}}(\cdot,P)$. Let us denote the uniaxial well of $\F$ by $$Z:=  \left\lbrace s_\star \left(m \otimes m - \f{1}{3}\I\right) : m\in \mathbb{S}^2 \right\rbrace.$$ First, as observed in \cite[Equation 1.14]{bpp}, we have the following inequality regarding $\F$ and $\dist(\cdot,Z)$, the Euclidean distance to $Z$: there exists $\delta>0$ and $C>0$ such that 
\begin{equation}\notag
\f{1}{C}\dist(Q, Z)^2 \leq \F(Q) \leq C \dist(Q,Z)^2
\end{equation}
for any $Q$ such that $\dist(Q,Z)<\delta$. We will say $\F \sim \dist(\cdot,Z)^2$ in a $\delta$--neighborhood of $Z$. It is quickly seen that after adding $2f_s = 2\gamma|(\I - \z \otimes \z )Q\z|^2$ to $\F$, we have $W \sim \dist(\cdot,P)^2$ in some $\delta'$--neighborhood of $P$. Since $W$ and $\dist(\cdot ,P)$ both vanish exactly on $P$, it follows that for any compact set $B\subset \s$, there exists $C>0$, depending on $B$, such that
\begin{equation}\label{quadratic}
 W \sim \dist(\cdot,P)^2
\end{equation}
on $ B$. It is straightforward to see using \eqref{quadratic} and the definition of $d_{\sqrt{W}}$ that $d_{\sqrt{W}}( \cdot, P)$ satisfies the same property, namely that on any compact set $B\subset \s$,
\begin{equation}\label{quadratic2}
d_{\sqrt{W}}(\cdot,P) \sim \dist(\cdot,P)^2.
\end{equation}
\par
We will also need, similar to Section \ref{gproof}, geodesics under the degenerate metric $d_{\sqrt{W}}$ along with the solutions of an ODE to construct the boundary layer. We aim to bridge the boundary data
$$
g(x)=-3\beta \left( n(x) \otimes n(x) - \f{1}{3} \I \right) 
$$
to the well $P_1$. The details are different than those in Section \ref{gproof} due to the need for precise estimates of the error involved in the construction. First, fix $x_0 \in \pa \Om$ and obtain a curve $\gamma_{x_0}$ which is a geodesic for $d_{\sqrt{W}}(g(x_0),P_1)$. We assume $\gamma_{x_0}:[0,b]\rightarrow \s$ which is parametrized with respect to arc length, so that $b$ is the Euclidean length of $\gamma$, and satisfies $\gamma_{x_0}(0)=g(x_0)$, $\gamma_{x_0}(b)\in P_1$, and 
\begin{equation}\notag
\int_0^b \sqrt{W(\gamma_{x_0}(t))}|\gamma_{x_0}'(t)| \,dt =\p_1(g(x_0)).
\end{equation}
By the assumption \eqref{assu} that $\p_1\left(g \right)<\p_2\left(g\right)$, we have $\p_1(g(x_0))=d_{\sqrt{W}}(P,g(x_0))$. We note also that \eqref{assu} allows us to require without loss of generality that \begin{equation}\label{van}W(\gamma_{x_0}(t))\textup{ vanishes only for }t=b.\end{equation}For $x\in \pa \Om$ with $x\neq x_0$, we define  $\gamma_{x}=r_{\theta_x} \gamma_{x_0} r_{\theta_x}^T$, where $\theta_x$ is chosen so that $r_{\theta_x} g(x_0) r_{\theta_x}^T= g(x)$.  We see that due to the rotational symmetry described at the beginning of Section \ref{gproof}, $\gamma_{x}(0)=g(x)$, $\gamma_{x}(b)\in P_1$, and 
\begin{equation}\label{ggg}
\int_0^b \sqrt{W(\gamma_{x}(t))}|\gamma_{x}'(t)| \,dt=\p_1(g(x)) = d_{\sqrt{W}}\left(P,g(x)\right)=d_{\sqrt{W}}(P,P_0).
\end{equation}
In addition, we remark that unlike in Section \ref{gproof}, the $\gamma_{x}$ vary smoothly in $x$ over all of $\pa \Om$, even near $x_0$.
\par
Next, we consider the following ODE:
\begin{equation}\label{odee}
 \begin{cases} 
      \f{\pa}{\pa s}h(s) = \sqrt{W(\gamma_{x_0}(h(s)))},  \\
      h(0) = 0.
   \end{cases}
\end{equation}
Using the fact that $W(\gamma_{x_0}(t))\geq C(t-b)^2$ for some $C>0$, which follows from \eqref{quadratic} and \eqref{van}, one can argue as in \cite[Equations (1.17)-(1.21)]{sternbergthesis} and conclude that the solution $h$ to \eqref{odee} is defined on $[0,\infty)$, increasing, and approaches $b$ exponentially as $s\rightarrow \infty$. We observe that $h$ also solves the same ODE with $\gamma_{x_0}$ replaced by $\gamma_{x}$ for any $x \in \pa \Om$ because of the rotational symmetry of $W$.\par
Let us now define the $R_\e$ near $\pa \Om$. We recall from Section \ref{gproof} the notation $\sigma(x)$ for the projection of $x$ onto $\pa \Om$, $d(x)$ for the distance to $\pa \Om$, and $\Om_\zeta$ for the set $\{x\in \Om : d(x)\geq \zeta\}$. We fix $\zeta>0$ such that $d$, $\sigma$ are $C^1$ on $\Om \setminus \Om_\zeta$ and define for small $\e$
\[
 R_\e(x) = \begin{cases} 
\gamma_{\sigma(x)}(b),   &  2\sqrt{\e}\leq d(x) \leq \zeta, \\
\gamma_{\sigma(x)}\left(h\left(\f{d(x)}{\e}\right)\right), & d(x)\leq \sqrt{\e}.
   \end{cases}
\]
For $x$ such that $\sqrt{\e} < d(x) < 2\sqrt{\e}$, we define $R_\e$ separately on each segment normal to $\{x:d(x)=\sqrt{\e}\}$ by linearly interpolating between the values of $R_\e$ at the intersections of the segment with $\{x:d(x)=\sqrt{\e}\}$ and $\{x:d(x)=2\sqrt{\e}\}$. We will define $R_\e$ on $\Om_\zeta$ at the end.\par
We will now estimate the energy of $R_\e$ on $\Om \setminus \Om_\zeta$ and prove that it contributes the leading order term in \eqref{t6} up to an error of $O(\e)$. For each $x \in \Om\setminus \Om_\zeta$, let $\tau=\tau(x)$ be a unit vector tangent to the level set of $d$ at $x$ and $\eta=\eta(x)$ be the unit vector $\nabla d(x)$ perpendicular to $\tau$. We write 
\begin{equation}\notag
\int_{\Om\setminus\Om_\zeta} \left(\e|\nabla R_\e|^2 + \f{1}{\e}W(R_\e) \right)\,dx\notag
= \int_{\Om\setminus\Om_\zeta} \left(\e\left|\f{\pa}{\pa \tau}R_\e\right|^2+\e\left|\f{\pa}{\pa \eta}R_\e\right|^2 + \f{1}{\e}W(R_\e)\right)\,dx.
\end{equation}
First, we have that $\left|\f{\pa}{\pa \tau}R_\e\right|$ is bounded by some constant $C$ depending on the Lipschitz constant of $g$ and independent of $\e$ on $\Om \setminus \Om_\zeta$, so that 
\begin{equation}\notag
\int_{\Om \setminus \Om_\zeta} \e\left|\f{\pa}{\pa \tau}R_\e\right|^2 \,dx \leq O(\e).
\end{equation}
Since $\e\left|\f{\pa}{\pa \tau}R_\e\right|^2$ is the only non--zero term in the integrand when $2\sqrt{\e}\leq d(x) \leq \zeta$, we have
\begin{equation}\notag
\int_{\Om\setminus\Om_\zeta} \left(\e|\nabla R_\e|^2 + \f{1}{\e}W(R_\e) \right)\,dx = \int_{\{x:d(x)<2\sqrt{\e}\}} \left(\e\left|\f{\pa}{\pa \eta}R_\e\right|^2 + \f{1}{\e}W(R_\e) \right)\,dx+O(\e).
\end{equation}
Next, using the fact that $h(s)$ approaches $b$ exponentially as $s\rightarrow \infty$, one can easily conclude as in \cite[Equation (2.26)]{sternbergthesis} that
\begin{equation}
\int_{\{x:\sqrt{\e}<d(x)<2\sqrt{\e}\}} \left(\e\left|\f{\pa}{\pa \eta}R_\e\right|^2 + \f{1}{\e}W(R_\e) \right)\,dx \leq O(\e),
\end{equation}
so that in order to capture the leading order term in \eqref{t6}, it remains to estimate
\begin{equation}\label{blech}
I:= \int_{\{x:d(x)<\sqrt{\e}\}} \left(\e\left|\f{\pa}{\pa \eta}R_\e\right|^2 + \f{1}{\e}W(R_\e) \right)\,dx
.\end{equation}\par
Our estimates for $I$ utilize a similar strategy as in \cite[Proposition 2.1]{as3}. In particular, the remarks \eqref{quadratic}, \eqref{quadratic2} along with the ODE \eqref{odee} will be crucial in obtaining a precise form of the error. We note that \eqref{quadratic} and \eqref{quadratic2} imply that $W\sim d_{\sqrt{W}}(\cdot,P)$ on any compact set $B$. With $B=\{R_\e(x) : x \in \Om \setminus \Om_\zeta\}$, we can write
\begin{equation}\label{simm}
W(R_\e) \sim d_{\sqrt{W}}(R_\e, P)=\p_1(R_\e)
\end{equation}
on $ \Om \setminus \Om_\zeta$.
We will also use the identity
\begin{equation}\notag
\sqrt{\e}\f{\pa}{\pa \eta}R_\e(x)= \f{1}{\sqrt{\e}}\sqrt{W(R_\e(x))}\gamma'_{\sigma(x)}\left( h \left( \f{d(x)}{\e}\right) \right)   = -\f{1}{\sqrt{\e}}(\nabla \p_1)(R_\e(x)) ,
\end{equation}
for $x$ such that $d(x)<\sqrt{\e}$, which follows from \eqref{odee} and the fact that as the gradient of a distance function, $(\nabla \p_1)(\gamma_{\sigma(x)})$ is parallel to the geodesic direction $\gamma_{\sigma(x)}'$. From this identity and the fact that $|\gamma_{\sigma(x)}'|=1$ we conclude that
\begin{equation}\label{e1}
\e\left|\f{\pa}{\pa \eta}R_\e\right|^2 = \sqrt{\e}\f{\pa}{\pa \eta}R_\e \cdot -\f{1}{\sqrt{\e}}(\nabla \p_1)(R_\e) = \f{1}{\e}W(R_\e).
\end{equation}
Now applying the coarea formula and \eqref{e1}, we write
\begin{align}\notag
I
&=2 \int_0^{\sqrt{\e}} \int_{\{x:d(x)=s\}} \f{\pa}{\pa \eta} R_\e \cdot (-\nabla \p_1) (R_\e) \,d\h(x) \,ds \\ \notag
&=2 \int_0^{\sqrt{\e}} \int_{\{x:d(x)=s\}}  - \f{\pa}{\pa \eta}( \p_1\circ R_\e) \,d\h(x) \,ds \\ \notag
&= 2 \int_0^{\sqrt{\e}} \int_{\{x:d(x)=s\}}  - \f{\pa}{\pa s}\left( \p_1\left( \gamma_{\sigma(x)}\left( h\left( \f{s}{\e} \right) \right) \right) \right) \,d\h(x) \,ds \\ \notag
&= 2 \int_0^{\sqrt{\e}} \int_{\pa \Om}  - \f{\pa}{\pa s}\left( \p_1\left( \gamma_{x}\left( h\left( \f{s}{\e} \right) \right) \right) \right) |J(P_t)(x)|\,d\h(x) \,ds, \\ \notag
\end{align}
where $J(P_t)(x)$ is the Jacobian of the map $P_t : \pa \Om \rightarrow \{x \in \Om : d(x)=t\}$. For ease of notation let us refer to $\p_1\left( \gamma_{x}\left( h\left( s\e^{-1} \right) \right) \right)$ as $(\p_1\circ R_\e)(x,s)$ for $x \in \pa \Om$ and $0 \leq s \leq \sqrt \e$. Recalling the estimate $|J(P_t)(x)-1|\leq Ct$ from \eqref{jac}, we can write
\begin{align}\notag
I
&\leq 2\int_0^{\sqrt{\e}} \int_{\pa \Om} - \left( \f{\pa}{\pa s} (\p_1 \circ R_\e)(x,s)\right)(1+Cs) \,d\h(x) \,ds \\ \label{blahhh}
&= 2\int_{\pa \Om} \int_0^{\sqrt{\e}}  - \left( \f{\pa}{\pa s} (\p_1 \circ R_\e)(x,s)\right) (1+Cs) \,ds \,d\h(x).
\end{align}
We estimate the inner integral with the goal of improving our estimate of $\int_{\Om\setminus \Om_\zeta}W(R_\e)\,dx$. For each $x\in \pa \Om$, we have
\begin{align}\label{goodie}
\int_0^{\sqrt{\e}}  -  &\left(\f{\pa}{\pa s}(\p_1 \circ R_\e)(x,s)\right)(1+Cs)\,ds \\ \notag
&= (\p_1 \circ R_\e)(x,0)-(\p_1 \circ R_\e)(x,\sqrt{\e})(1+C\sqrt{\e}) + \int_0^{\sqrt{\e}}(\p_1 \circ R_\e)(x,s)C \,ds \\ \notag
&\leq (\p_1 \circ R_\e)(x,0) + \int_0^{\sqrt{\e}}(\p_1 \circ R_\e)(x,s)C \,ds \\ \notag
&\leq C
\end{align}
for some constant $C$ independent of $x$ and $\e$.  This estimate on the inner integral then yields $I\leq C$ after substituting into \eqref{blahhh}, which implies
\begin{equation}\label{w}
\int_{\{x:d(x)<\sqrt{\e}\}} W(R_\e) \,dx \leq C\e.
\end{equation}
But since $W(R_\e) \sim \p_1(R_\e)$ by \eqref{simm}, we see from \eqref{w} that 
\begin{equation}\label{bddi}
\int_{\{x:d(x)<\sqrt{\e}\} } \p_1\left( R_\e \right) \,dx  \leq C\e.
\end{equation}
Finally, integrating \eqref{goodie} over $x \in \pa \Om$ and using \eqref{bddi}, we have
\begin{align}\label{final}
I_1 &\leq 2\int_{\pa \Om} (\p_1\circ R_\e)(x,0) \,d\h(x) + 2C\int_{\{x:d(x)<\sqrt{\e}\} } \p_1 ( R_\e) \,dx \\ \notag &= 2\int_{\pa \Om} \p_1(g(x)) \,d\h(x) + O(\e) .
\end{align}
From \eqref{blech} and \eqref{final}, we conclude
\begin{equation}\label{layer}
\int_{\Om\setminus\Om_\zeta} \left(\e|\nabla R_\e|^2 + \f{1}{\e}W(R_\e) \right)\,dx\leq 2\int_{\partial \Omega} \p_1(g(x)) \, d\h(x) +O(\e).
\end{equation}
\par
Next, we define the $R_\e$'s on $  \Om_\zeta$. We will use a construction from \cite{bpp}, in which the authors examine the functionals
\begin{equation}\notag
\int_A (f_e(Q)+\e^{-2}f_b(Q))\,dx
\end{equation}
among $Q$-tensors which have $\z$ as an eigenvector and $P_1$--valued boundary data $\tilde{g}$ with degree $\tilde{k}$. Here $A\subset \R^2$, $f_e$ is the general Landau-de Gennes elastic energy density, and $f_b$ is a general bulk energy density which includes $\F$ as a specific example. In the proof of \cite[Lemma 3.6]{bpp}, a sequence of $Q$-tensors is constructed with energies bounded by $s_\star^2 \pi \tilde{k} \log\f{1}{\e}+O(1)$. The interested reader can find the sequence, denoted by $(\textbf{w}',r')$, in the proof in \cite[p. 810]{bpp}. Since we defined $R_\e$ on $\Om \setminus \Om_\zeta$ so that $R_\e(\pa \Om_\zeta)=\{\gamma_{\sigma(x)}(b):x\in \pa \Om_\zeta \} \subset P_1$ and  $R_\e$ restricted to $\pa \Om_\zeta$ has degree $k$ in the sense of \eqref{degreee}, we can apply the result of \cite[Lemma 3.6]{bpp}. We obtain $R_\e : \Om_\zeta \rightarrow \s$ such that $\z$ is an eigenvector for each $R_\e(x)$ and
\begin{align}\label{tilde}
\int_{\Om_\zeta}\left( \e|\nabla R_\e|^2 + \f{W(R_\e)}{\e}\right)\,dx   &= \int_{\Om_\zeta}\left( \e|\nabla R_\e|^2 + \f{\F(R_\e)}{\e}\right)\,dx  \\ \notag &\leq \e s_\star^2\pi k \log \f{1}{\e} + O(\e).
\end{align}
Adding the estimates \eqref{layer} and \eqref{tilde} gives
\begin{equation}\notag
F_\e(R_\e) \leq 2 \int_{\partial \Omega} \p_1(g(x)) \, d\h(x) + \e s_\star^2\pi k \log\f{1}{\e} + O(\e),
\end{equation}
and the upper bound is proved.
\end{proof}\par\color{black}
\begin{proof}[Proof of the lower bound.]
We turn now to the proof of the lower bound
\begin{equation}\label{p3.1}
F_\e(Q_\e) \geq 2 \int_{\partial \Omega} \p_1(g(x)) \, d\h(x) + \e s_\star^2\pi k \log\f{1}{\e} + O(\e).
\end{equation}
Similar to the proof of the lower semicontinuity condition \eqref{lsc} in Section \ref{gproof}, we want to replace the integrals of $f_s$ over the top and bottom of the cylinder by the integral of $2f_s$ over $\Om \times (0,1)$. By the calculations \eqref{decay1} and \eqref{decay2}, we have 
\begin{equation}\notag
F_{{\e}}(Q_{\e})=\int_{\Omega\times(0,1)} \left({\e}|\nabla_xQ_{\e}|^2 + \f{|\nabla_{z} Q_{\e}|^2}{ {\e} ^3} + \f{1}{{\e}}W(Q_{\e})\right)\,dx\,dz+O(\e).
\end{equation}
Next, using the fact that the boundary data $g$ is independent of $z$, we further estimate the energy $F_\e(Q_\e)$ from below. We write
\begin{align}\notag
F_{{\e}}(Q_{\e})&\geq \int_{\Omega\times(0,1)} \left({\e}|\nabla_xQ_{\e}|^2 + \f{1}{{\e}}W(Q_{\e})\right)\,dx\,dz+O(\e)\\ \notag
&=\int_0^1 \int_{\Omega} \left({\e}|\nabla_xQ_{\e}|^2 + \f{1}{{\e}}W(Q_{\e})\right)\,dx \,dz+O(\e)\\ \notag
&\geq \int_0^1 \min_{Q\in H_g^1(\Om;\s)}\left\lbrace\int_{\Omega} \left({\e}|\nabla_xQ|^2 + \f{1}{{\e}}W(Q)\right)\,dx \right\rbrace\,dz+O(\e)\\ \notag
&= \min_{Q\in H_g^1(\Om;\s)}\left\lbrace\int_{\Omega} \left({\e}|\nabla_xQ|^2 + \f{1}{{\e}}W(Q)\right)\,dx \right\rbrace+O(\e).
\end{align}
Let $G_\e$ be the functional
$$G_\e(Q):=\int_{\Omega} \left({\e}|\nabla_xQ|^2 + \f{1}{{\e}}W(Q)\right)\,dx$$
defined on $H^1_g(\Om;\s)$ and denote its minimizer by $\tilde{Q}_\e$. We have now that
\begin{equation}\notag
F_\e(Q_\e) \geq G_\e(\tilde{Q}_\e)+O(\e).
\end{equation}
It then suffices to show that
\begin{equation}\label{bound}
G_\e(\Q) \geq 2 \int_{\partial \Omega} \p_1(g(x)) \,d\h(x) + \e s_\star^2  \pi k \log \f{1}{\e} + O(\e).
\end{equation}\par
The proof of \eqref{bound} follows the proof of \cite[Proposition 3.1]{as3}, so we omit the details. We remark that the main difficulty in the proof is separating the leading order boundary layer contribution of order $O(1)$ from the interior contribution of order $O\left(\e\log \f{1}{\e}\right)$. In the case of Ginzburg--Landau, this was accomplished in \cite{as1} by using a decomposition formula from \cite{lm}. For a general potential $W$, an analogous formula cannot be expected, so different techniques are needed. The general program is to consider a neighborhood $\Om \setminus \Om_{c\e^\alpha}$ of $\pa \Om$ for $\alpha<1$ and prove that $G_\e(\Q,\Om \setminus \Om_{c\e^\alpha})$ contributes the leading order term $2\int_{\partial \Omega} \p_1(g(x)) \,d\h(x)$ up to a certain error term $e_\e$; see  \cite[Proposition 3.2]{as3}. One then proceeds by showing that $G_\e(\Q, \Om_{c\e^\alpha})+e_\e$ is bounded below by $\e s_\star^2  \pi k \log \f{1}{\e}$; see \cite[Proposition 5.1]{as2}, which adapts techniques from \cite{bbh,struwe}.\end{proof}
With the asymptotic development of \cref{asy} in hand, a natural question is whether a subsequence of the minimizers $Q_\e$ converges to some $Q_0$. Results of this type have been obtained for the case of Ginzburg--Landau type problems in \cite{as1,as2,as3,bbh} and for a two--dimensional Landau--de Gennes model in \cite{bpp}. The limiting map for the Ginzburg--Landau type problems is the so--called ``canonical harmonic map" identified in \cite{bbh} and generalized in \cite{as1,as2,as3}. By utilizing a suitable change of variables, a similar result is proved for the two--dimensional Landau--de Gennes model in \cite{bpp}. In the latter case, the limiting map is uniaxial and the minimizers along a convergent subsequence have degree $1/2$ around the singularities of the limiting map, cf. \cite[Corollary A]{bpp}. We expect a similar result to hold for our problem, but we have not carried out the details. Finally, one could certainly determine the location of the singularities of the limiting map, corresponding to the disclination lines in the nematic film. In all of the above works, the locations of the singularities are governed by the minimization of a ``renormalized energy," and such a program could be carried out for this problem as well.
\section{Appendix}\label{app}
\begin{lemma}\label{i}
The minimum of $W=\F+2f_s$ is achieved and can be characterized as follows:
\begin{enumerate}
\item If $\beta$ cannot be written as a convex linear combination of $\lambda_i$'s, where $\lambda=(\lambda_1,\lambda_2,\lambda_3)$ is a stationary point of $\F$, then $\z$ is an eigenvector of any minimizer.
\item If $\beta$ can be written as a convex linear combination of $\lambda_i$'s, then there are two cases:
\begin{enumerate}[(i)]
\item if $\beta=\lambda_i$ for some $i$, so that the convex linear combination is trivial, then $\z$ is an eigenvector of any minimizer;
\item if $\beta$ is not equal to one of the $\lambda_i$'s, so the convex linear combination is non--trivial, then $\z$ is not a eigenvector of any minimizer.
\end{enumerate}
\end{enumerate}
In either case $(1)$ or $(2)$, minimizers may be isotropic, uniaxial, or biaxial.
\end{lemma}
\begin{proof} The quartic growth of $\F$ at infinity and the fact that $f_s$ is positive immediately yield the existence of a global minimizer for $W$. Next, we state the equations for stationary points of $\F(Q)$, expressed as $\F(\lambda)$, subject to the condition $\sum \lambda_i=0$. We obtain the system
\begin{equation}\notag
\begin{cases} 
      (\F)_{\lambda_i}+k=0\\ \sum \lambda_i=0,\\
   \end{cases}
\end{equation}
where $k$ is a Lagrange multiplier. We are interested in finding critical points for $W(Q)=\F(\lambda)+\alpha(Q \z \cdot \z-\beta)^2+\gamma|(\I - \z \otimes \z )Q\z|^2$, where $\lambda=(\lambda_1,\lambda_2,\lambda_3)$ is the set of eigenvalues for $Q$. Let us write $Q=\lambda_i v_i \otimes v_i$, where $\{v_i\}$ is a mutually orthonormal set of vectors in $\R^3$. First, we consider the case where $\gamma=0$. Rewriting $W$ using the $\lambda_i$'s and $v_i$'s, we obtain
\begin{equation}\notag
W(Q) = \F(\lambda) + \alpha(\lambda_i (v_i \cdot \z)^2-\beta)^2.
\end{equation}
Note that since $1=|\z|^2 = \sum (v_i\cdot \z)^2$, if we let $y_i=(v_i\cdot \z)^2$, then $\sum y_i = 1$ and $y_i\geq 0 $ for each $i$. Minimizing $W$ subject to the constraint $\sum\lambda_i=0$ is therefore equivalent to minimizing
\begin{equation}\notag
\tilde{W}(\lambda,y) = \F(\lambda)+\alpha(\lambda_i y_i - \beta)^2
\end{equation}
subject to the constraints $\sum \lambda_i =0$, $\sum y_i = 1$, and $y_i \geq 0 $ for each $i$. 
We define the auxiliary function
\begin{equation}
F(\lambda,y)=\F(\lambda)+\alpha(\lambda \cdot y -\beta)^2 + h_\lambda\sum \lambda_i + h_y(\sum y_i-1),
\end{equation}
where $h_\lambda$ and $h_y$ are Lagrange multipliers. Any minimizer of $\tilde{W}$ must be a stationary point of $F$. We calculate
\begin{equation}\label{One}
\nabla_\lambda F = \nabla_\lambda \F(\lambda)+2\alpha(\lambda \cdot y-\beta)y+h_\lambda(1,1,1) = 0
\end{equation}
and
\begin{equation}
\label{Two} \nabla_y F=2\alpha(\lambda\cdot y -\beta)\lambda+h_y(1,1,1)=0.
\end{equation}
Adding the components of both sides of \eqref{Two} and using $\sum \lambda_i=0$, we find that $h_y=0$. Doing the same for \eqref{One} gives
\begin{equation}\notag
\sum(\F)_{\lambda_i}(\lambda)+2\alpha(\lambda \cdot y - \beta) + 3 h_\lambda=0,
\end{equation}
so we solve for $h_\lambda$ and get
\begin{equation}\notag
h_\lambda = -\f{1}{3}\sum(\F)_{\lambda_i}(\lambda) - \f{2}{3}\alpha(\lambda \cdot y - \beta)
.\end{equation}
Using this expression for $h_\lambda$ in \eqref{One}, we see that we need to solve
\begin{equation}
(\F)_{\lambda_i}(\lambda)+2\alpha(\lambda \cdot y-\beta)(y_i-\f{1}{3})-\f{1}{3}\sum(\F)_{\lambda_i}(\lambda) = 0
\end{equation}
and
\begin{equation}\notag
2\alpha(\lambda \cdot y - \beta)\lambda_i=0
\end{equation}
subject to the constraints $\sum y_i=1$, $y_i\geq 0$, and $\sum \lambda_i=0$. If there exists $\lambda$ and $y$ such that $\lambda \cdot y-\beta=0$, then the second equation is automatically satisfied and the first equation reduces to \eqref{One}. Hence any critical point $(\lambda,y)$ of $W$ which satisfies $\lambda \cdot y - \beta=0$ must also be a critical point of $\F$. For example, if $\overline{\lambda}=(\overline{\lambda}_1,\overline{\lambda}_2,\overline{\lambda}_3)$ is the minimizer of $\F$ and $\beta$ is equal to $\sum \overline{y}_i \overline{\lambda}_i$ for some $0\leq \overline{y}_i \leq 1$, then setting $y_i=\overline{y}_i$ and and $\lambda=\overline{\lambda}$ yields a minimizer for $W$, since 
\begin{equation}\notag
\F(\overline{\lambda}) \leq \min W(\lambda,y) \leq W(\overline{\lambda},\overline{y})=\F(\overline{\lambda}).
\end{equation}
If $\beta=\overline{\lambda}_i$, so the convex linear combination is the trivial one with the corresponding $\overline{y}_i=1$, then we see from the definition of $y_i$ that $\z$ is an eigenvector for the minimizer. Conversely, if the convex linear combination is non--trivial, so that at least two of the $y_i$'s are non--zero, we see that it is impossible that $\z$ is an eigenvector. Note that the minimizer is uniaxial in this case if the minimizer of $\F$ is uniaxial.\par
Suppose now that $\lambda\cdot y-\beta \neq 0$. Then $\lambda_i=0$ for each $i$; since $(\F)_\lambda(0)=0$, \eqref{One} is satisfied by setting $y_i=1/3$ for each $i$. Note in this case the minimizer is isotropic and so of course $\z$ is an eigenvector. Any other stationary points of $\F$, and thus extrema, must occur on the boundary of the admissible set in $y$, which happens when any one of the $y_i$ is $0$ or $1$. By symmetry, we only analyze the cases when $y_3$ is $0$ or $1$.\par
Suppose that $y_3=0$. We need to find critical points for
\begin{equation}\notag
F(y_1,y_2,\lambda)=\F(\lambda)+\alpha(y_1 \lambda_1 + y_2 \lambda_2-\beta)^2+h_y(y_1+y_2-1)+h_\lambda\sum \lambda_i,
\end{equation}
subject to the constraints $y_1+y_2=1$, $\sum\lambda_i=0$. Hence, we have
\begin{equation}\notag
\begin{cases}
(\F)_{\lambda_1}+2\alpha(y_1 \lambda_1 + y_2\lambda_2-\beta)y_1+h_\lambda=0,\\
(\F)_{\lambda_2}+2\alpha(y_1 \lambda_1 + y_2\lambda_2-\beta)y_2+h_\lambda=0,\\
(\F)_{\lambda_3}+h_\lambda=0,\\
2\alpha(y_1\lambda_1+y_2\lambda_2-\beta)\lambda_1+h_y=0,\\
2\alpha(y_1\lambda_1+y_2\lambda_2-\beta)\lambda_2+h_y=0,\\
y_1+y_2=1,\textup{ } \sum \lambda_i=0.
\end{cases}
\end{equation}
If 	$\beta=y_1\lambda_1+y_2\lambda_2$ for a solution, we argue as before and see that such a critical point must be a critical point of $\F$. If not, the fourth and fifth equations show that $\lambda_1=\lambda_2:=\overline{\lambda}$ and the first and second equations along with the symmetry of $\F$ with respect to permutations of the $\lambda_i$'s show that $y_1=y_2=1/2$. It follows that $\overline{\lambda}$ must solve
\begin{equation}\notag
(\F)_{\lambda_1}(\overline{\lambda},\overline{\lambda},-2\overline{\lambda})-(\F)_{\lambda_3}(\overline{\lambda},\overline{\lambda},-2\overline{\lambda})+\alpha(\lambda-\beta)=0.
\end{equation}
Note that the solution is uniaxial. Also note that $y_3$ must be $0$, which implies that $v_3$, the eigenvector associated to $\lambda_3$, must be perpendicular to $\z$. Hence $\z$ is in the eigenplane spanned by $v_1$ and $v_2$ and is thus an eigenvector.\par
Now suppose that $y_2=y_3=0$, so that $y_1=1$. We are looking for critical points of
\begin{equation}\notag
\F(\lambda)+\alpha(\lambda_1-\beta)^2+h_\lambda \sum \lambda_i
\end{equation}
subject to the constraint $\sum \lambda_i=0$. Suppose that $\lambda_1 \neq \beta$; if it was, we argue as before. We need to solve
\begin{equation}\notag
\begin{cases}
(\F)_{\lambda_1}+2\alpha(\lambda_1-\beta)+h_\lambda=0,\\
(\F)_{\lambda_2}+h_\lambda=0,\\
(\F)_{\lambda_3}+h_\lambda=0,\\
\sum \lambda_i=0.
\end{cases}
\end{equation}
It follows that $\lambda_1$, $\lambda_2$ must solve:
\begin{equation}\notag
\begin{cases}
(\F)_{\lambda_1}(\lambda_1,\lambda_2,-\lambda_1-\lambda_2)-(\F)_{\lambda_3}(\lambda_1,\lambda_2,-\lambda_1-\lambda_2)+2\alpha(\lambda_1-\beta)=0,\\
(\F)_{\lambda_2}(\lambda_1,\lambda_2,-\lambda_1-\lambda_2)=(\F)_{\lambda_3}(\lambda_1,\lambda_2,-\lambda_1-\lambda_2).
\end{cases}
\end{equation}
Clearly, the second equation is always satisfied when the second and third eigenvalue are the same, so that $\lambda_1=-2\lambda_2$, and the corresponding critical point is uniaxial. It is also possible that the second equation holds for other, biaxial, choices of $\lambda$'s as numerics seems to indicate. Regardless, $\z$ is an eigenvector of the minimizing $Q$ in this case, since $y_3=1$.\par
We have shown now that $\z$ is an eigenvector for any minimizer of $W$ when $\gamma$ is $0$ and $\beta$ is not a convex linear combination of the eigenvalues $\lambda_i$ of a stationary point of $\F$. This also remains true under the same assumption on $\beta$ if $\gamma$ is positive, since any minimizer of $W$ when $\gamma$ is $0$ which has $\z$ as an eigenvector must also minimize $W$ for any positive $\gamma$. In the case where $\beta$ is such a linear combination, we are able to construct examples in which $\z$ is and is not an eigenvector for the corresponding minimizers of $W$ with $\gamma=0$. Such examples should also be possible with positive $\gamma$, depending perhaps on the relationship between $\gamma$ and $\alpha$.\end{proof}
\section*{Acknowledgments} The author is grateful to his advisers, Dmitry Golovaty and Peter Sternberg, for their invaluable advice. He also acknowledges partial support from the NSF through grants DMS-1101290 and DMS-1362879 and from the Department of Mathematics at Indiana University.
\bibliographystyle{siam} 

\bibliography{Main}
\end{document}